\documentclass[11pt,twoside]{amsart}
\usepackage{amsmath, amsthm, amsfonts, amssymb, amscd}
\usepackage{graphicx}
\usepackage[ruled,vlined,norelsize]{algorithm2e}
\usepackage{caption}
\usepackage{subcaption}
\usepackage{tikz}
\usepackage{hyperref}
\usepackage{comment}

\theoremstyle{plain}
\newtheorem{theorem}{Theorem}[section]
\newtheorem{lemma}[theorem]{Lemma}
\newtheorem{proposition}[theorem]{Proposition}
\newtheorem{corollary}[theorem]{Corollary}

\theoremstyle{definition}
\newtheorem{definition}[theorem]{Definition}
\newtheorem{example}[theorem]{Example}

\theoremstyle{remark}
\newtheorem{remark}[theorem]{Remark}

\DeclareMathOperator*{\argmax}{arg\,max}
\providecommand{\set}[1]{\{#1\}}
\providecommand{\abs}[1]{\left|#1\right|}
\providecommand{\norm}[1]{\left\lVert#1\right\rVert}
\newcommand{\field}[1]{\mathbb{#1}}
\newcommand{\R}{\field{R}}

\newcommand{\prj}{\operatorname{proj}}
\renewcommand{\ker}{\operatorname{ker}}
\newcommand{\im}{\operatorname{im}}

\title{Random Walks on Simplicial Complexes and Harmonics}
\author{Sayan Mukherjee}
\address{Sayan Mukherjee, Departments of Statistics, Mathematics, and Computer Science, Duke University}
\author{John Steenbergen}
\address{John Steenbergen, Department of Mathematics, Duke University}
\date{\today}

\begin{document}

\begin{abstract}
In this paper, we introduce random walks with absorbing states on simplicial complexes.  Given a simplicial complex of dimension $d$, a random walk with an absorbing state is defined which relates to the spectrum of the $k$-dimensional Laplacian for $1 \leq k \leq d$ and which relates to the local random walk on a graph defined by Fan Chung.  We also examine an application of random walks on simplicial complexes to a semi-supervised learning problem. Specifically, we consider a label propagation algorithm on oriented edges, which applies to a generalization of the partially labelled classification problem on graphs.
\end{abstract}

\maketitle

\section{Introduction}
\subsection{Background}

The relation between spectral graph theory and random walks on graphs has been well studied and has both theoretical and practical implications 
\cite{chung1997spectral,Lovasz96,MeilaShi01}. A classic example of this relation is graph expansion (see \cite{hoory2006expander}). 
Loosely speaking, graph expansion measures how far a graph is from being disconnected (i.e., having a nontrivial reduced $0$-th homology class).  The two common characterizations of graph expansion use either the Cheeger number which relates to spectral graph theory or the mixing time of a random walk on the graph.

In this paper we examine an analagous relation between random walks on simplicial complexes and spectral properties of
higher order Laplacians. A simplicial complex is a higher-dimensional generalization of a graph consisting of vertices and edges as well as higher-dimensional simplices such as triangles and tetrahedra.  The graph Laplacian was generalized to simplicial complexes by Eckmann \cite{eckmann1944harmonische}, resulting in what are called higher order combinatorial Laplacians. The $k$-th order combinatorial Laplacian, or $k$-Laplacian, can be used to study expansion in the sense that the spectrum of the $k$-Laplacian provides information on how far from the complex is from having a nontrivial $k$-th (co)homology class. The graph Laplacian is simply the $0$-th order combinatorial Laplacian.  There has been recent work extending Cheeger numbers and random walks to higher dimensions \cite{dotterrer2012coboundary, lubotzky2013ramanujan, parzanchevski2012isoperimetric, parzanchevski2012simplicial, steenbergen2012cheeger}. 

The $k$-Laplacian is naturally decomposed into two parts commonly called the up $k$-Laplacian and the down $k$-Laplacian. The graph case is an exception in that there is only an up $0$-Laplacian; the down $0$-Laplacian is the zero matrix. This fact suggests that a straightforward generalization of the theory of graph expansion to higher dimensions may only relate to the up $k$-Laplacian.  Indeed, the Cheeger number of a graph was initially generalized so as to relate to the up $k$-Laplacian \cite{dotterrer2012coboundary}, with the generalization to the down $k$-Laplacian following soon after \cite{steenbergen2012cheeger}.

This decomposition also appears when studying random walks on simplicial complexes. In a recent paper, Rosenthal and Parzanchevski 
\cite{parzanchevski2012simplicial} generalized random walks on graphs to random walks on simplicial complexes.  They defined a Markov chain on the space of oriented $k$-simplexes that reflects the spectrum of the up $k$-Laplacian, assuming $0 \leq k \leq d-1$ where $d$ is the dimension of the simplicial complex.  The walk traverses the simplicial complex by moving between oriented $k$-simplexes via shared  $(k+1)$-simplexes.  In this paper we define a random walk that traverses the simplicial complex by traveling through shared $(k-1)$-simplexes. We demonstrate that this random walk is related to the spectrum of the down $k$-Laplacian and reflects the dimension of the $k$-th homology group over $\R$, assuming $1 \leq k \leq d$.  We also discuss the possibility of defining other random walks on simplicial complexes, including random walks relating to the full $k$-Laplacian and weighted Laplacians. We also apply random walks on simplicial complexes to a semi-supervised learning problem, propagating labels on edges. This generalizes the semi-supervised learning idea of propagating labels on nodes.

\subsection{Motivation}

We have two motivations for studying the random walk corresponding to the down Laplacian. The first motivation comes from an example.  	Consider the 2-dimensional simplicial complex formed by a hollow tetrahedron (or any triangulation of the 2-sphere).  We know that the complex has nontrivial 2-dimensional homology since there is a void.  However, this homology cannot be detected by the random walk defined in \cite{parzanchevski2012simplicial}, because there are no tetrahedrons that can be used by the walk to move between the triangles.  In general, the walk defined in \cite{parzanchevski2012simplicial} can detect homology from dimension 0 to co-dimension 1, but never co-dimension 0.  Hence, a new walk which can travel from triangles to triangles through edges is needed.
 
The second motivation relates to the geometry of random walks or diffusions and manifolds. The geometry captured by the graph Laplacian
as well as the Cheeger number and random walks on the graph have direct connections to the geometry of a manifold with  Neumann boundary
conditions. We will examine random walks that have connections to the geometry of a manifold with Dirichlet boundary conditions, denoted as ``Dirichlet'' random walks. Work by Fan Chung in \cite{chung2007random} has shown that there are alternative notions of the Laplacian and random walks on graphs that capture a Dirichlet-flavored geometry of graphs.  The definition of the ``local'' Cheeger number of a graph given in \cite{chung2007random} bears a striking resemblance to the definition of the Cheeger number of a manifold with Dirichlet boundary \cite{cheeger1970lower}. Also defined in \cite{chung2007random} is a ``local'' random walk that satisfies a Dirichlet boundary condition. In contrast, the usual random walk on a graph might be called Neumann.  The random walk defined by Rosenthal and Parzanchevski \cite{parzanchevski2012simplicial} generalizes the Neumann random walk to higher dimensions on simplicial complexes. In this paper we generalize the Dirichlet random walk.

\subsection{Summary of Results}
In this section we give a short summary of the main results.  Precise definitions of the terms used are given in section \ref{sec: Def}.  

In section \ref{sec: Main} we define a $p$-lazy Dirichlet random walk on the oriented $k$-simplexes of a $d$-dimensional simplicial complex $X$, where $1 \leq k \leq d$. This walk has a corresponding probability transition matrix $P$. In most analyses of random walks the questions of interest are convergence and rates of convergence of $\lim_{n \rightarrow \infty}P^n \nu  = \pi,$ where $\nu$ is the initial probability distribution on the states, $P^n \nu$ is the marginal distribution after $n$ steps of the walk, and $\pi$ is the stationary or invariant distribution.  For the usual random walk on a graph, the graph Laplacian is used to study the limiting behavoir of $P^n \nu$.  For the random walks we consider, orientation issues prevent a straightforward connection between the $k$-Laplacian and $P^n \nu$.  Instead, we find a connection between the $k$-Laplacian and $CTP^n \nu$ where $C$ is a constant and $T$ is a linear transformation.  The linear transformation $T$ enforces antisymmetry between the opposite orientations of a simplex.  Denoting $\sigma_+$ and $\sigma_-$ as the (arbitrarily chosen) positive and negative orientations of a simplex $\sigma$, $TP^n \nu$ is a function on the set of positively oriented simplexes such that
\[T  P^n \nu(\sigma_+) = P^n \nu(\sigma_+) - P^n \nu(\sigma_-).\]
The constant $C$ is a normalizing constant that ensures $CTP^n \nu$ has nontrivial limiting behavior.  Letting $M$ denote the maximum 
number of $k$-simplexes any $(k-1)$-simplex is contained in,
\[ C = \frac{M-1}{p(M-2)+1}. \]
Let $\textbf{1}_{\tau}$ denote the initial distribution supported on the oriented simplex $\tau$ and let $\widetilde{\mathcal{E}}^\tau_n := CTP^n \textbf{1}_{\tau}$. The down $k$-Laplacian is $L_k^{\text{down}} = \delta^{k-1}\partial_k$ where $\delta$ is a coboundary operator and $\partial$ is the boundary operator, and let $\lambda_k$ denote the smallest eigenvalue of $L_k^\text{down}$ with eigenvector perpendicular to $\im \partial_{k+1}$. The following proposition is a direct result of Theorem \ref{thm: Main}. 

\begin{proposition}
If $\frac{M-2}{3M-4}<p<1$, then the limit $\widetilde{\mathcal{E}}^\tau_\infty := \lim_{n \to \infty} \widetilde{\mathcal{E}}^\tau_n$ exists for all initial $\tau$.  In this case, the $k$-th homology group of $X$ with
coefficients in $\R$ is trivial if and only if $\widetilde{\mathcal{E}}^\tau_\infty \in \im \partial_{k+1}$ for all $\tau$.  In addition, if $p \geq \frac{1}{2}$ then
\[ \norm{\widetilde{\mathcal{E}}_n^{\tau} - \widetilde{\mathcal{E}}_\infty^{\tau}}_2  = O\left(\left[1-\frac{1-p}{(p(M-2)+1)(k+1)}\lambda_k\right]^n\right). \]
\end{proposition}
One difference in the above result with standard results on Markov chains is that the limiting object provides information on the homology of $X$.
This will be discussed further in section \ref{sec: Main}. Another difference is that for a connected graph the random walk is
irreducible, and the limit distribution is independent of the initial distribution.  In higher dimensions, this independence is lost, even for complexes with trivial $k$-th homology over $\R$.

\subsection{Related Work}
Both \cite{dodziuk1984difference,chung2007random} have examined the relation between graph random walks and the geometry of graphs with  Dirichlet
boundary conditions. In section \ref{sec: ex} we show that under certain conditions the Dirichlet random walk in codimension 0 coincides with the notion of a random walk on a graph with Dirichlet boundary.  A natural question to ask concerning random walks on simplicial complexes is: what would be the analogous
process on manifolds? In general we are not aware of results on the continuum limit of these walks. However, the Dirichlet random walk in codimension zero is analogous to the concept of Brownian motion with killing as described by Lawler and Sokal in \cite{lawlersocal88}.

\section{Definitions}
\label{sec: Def}

In this section we define the simplicial complex $X$, the chain and cochain complexes, and the $k$-Laplacian.

\subsection{Simplicial Complexes}
By a simplicial complex we mean an abstract finite simplicial complex.  Simplicial complexes generalize the notion of a graph to higher dimensions.  Given a set of vertices $V$, any nonempty subset $\sigma \subseteq V$ of the form $\sigma = \set{v_0, v_1, \ldots, v_j}$ is called a $j$-dimensional simplex, or $j$-simplex.  A simplicial complex $X$ is a finite collection of simplexes of various dimensions such that $X$ is closed under inclusion, i.e., $\tau \subseteq \sigma$ and $\sigma \in X$ implies $\tau \in X$.  While we will not need it for this paper, one can include the empty set in $X$ as well (thought of as a $(-1)$-simplex).  Given a simplicial complex $X$, denote the set of $j$-simplexes of $X$ as $X^j$.  We say that $X$ is $d$-dimensional or that $X$ is a $d$-complex if $X^d \neq \emptyset$ but $X^{d+1} = \emptyset$.  Graphs are 1-dimensional simplicial complexes.  We will assume throughout that $X$ is a $d$-complex for some fixed $d \geq 1$.

If $\sigma \in X^j$ and $\tau \in X^{j-1}$ and $\tau \subset \sigma$, then we call $\tau$ a \textit{face} of $\sigma$ and $\sigma$ a \textit{coface} of $\tau$.  Every $j$-simplex has exactly $j+1$ faces but may have any number of cofaces.  Given $\sigma \in X^j$ we define $\deg(\sigma)$ (called the \textit{degree} of $\sigma$) to be the number of cofaces of $\sigma$.  Two simplexes are \textit{upper adjacent} if they share a coface and \textit{lower adjacent} if they share a face.  The number of simplexes upper adjacent to a $j$-simplex $\sigma$ is $(j+1) \cdot \deg(\sigma)$ while the number of simplexes lower adjacent to $\sigma$ is $\sum_{\tau \subset \sigma} (\deg(\tau) - 1)$ where the sum is over all faces $\tau$ of $\sigma$.

Orientation plays a major role in the geometry of a simplicial complex.  For $j > 0$, an orientation of a $j$-simplex $\sigma$ is an equivalence class of orderings of its vertices, where two orderings are equivalent if they differ by an even permutation.   Notationally, an orientation is denoted by placing one of its orderings in square brackets, as in $[v_0, \ldots, v_j]$.  Every $j$-simplex $\sigma$ has two orientations which we think of as negatives of each other.  We abbreviate these two orientations as $\sigma_+$ and $\sigma_- = - \sigma_+$ (which orientation $\sigma_+$ corresponds to is chosen arbitrarily).  For $j=0$ there are no distinct orderings, but it is useful to think of each vertex $v$ as being positively oriented by default (so, $v_+ = v$) and having an oppositely-oriented counterpart $v_- := -v$.  For any $j$, we will use $X^j_+ = \set{ \sigma_+ : \sigma \in X^j}$ to denote a choice of positive orientation $\sigma_+$ for each $j$-simplex $\sigma$.  The set of all oriented $j$-simplexes will be denoted by $X^j_\pm$, so that $X^j_\pm = \set{\sigma_{\pm} : \sigma_+ \in X^j_+}$ and $|X^j_\pm| = 2|X^j|$ for any choice of orientation $X^j_+$.

An oriented simplex $\sigma_+ = [v_0, \ldots, v_j]$ induces an orientation on the faces of $\sigma$ as $(-1)^i[v_0, \ldots, v_{i-1},v_{i+1},\ldots,v_j]$.  Conversely, an oriented face $(-1)^i[v_0, \ldots, v_{i-1},v_{i+1},\ldots,v_j]$ of $\sigma$ induces an orientation $\sigma_+ = [v_0, \ldots, v_j]$ on $\sigma$.  Two oriented $j$-simplexes $\sigma_+$ and $\sigma'_+$ are said to be \textit{similarly oriented}, and we write $\sigma_+ \sim \sigma'_+$, if $\sigma$ and $\sigma'$ are distinct, lower adjacent $j$-simplexes and $\sigma_+$ and $\sigma'_+$ induce the opposite orientation on the common face (if $\sigma$ and $\sigma'$ are upper adjacent as well, this is the same as saying that $\sigma_+$ and $\sigma'_+$ induce the same orientation on the common coface).  If they induce the same orientation on the common face, then we say they are \textit{dissimilarly oriented} and write $\sigma_- \sim \sigma'_+$.  We say that a $d$-complex $X$ is \textit{orientable} if there is a choice of orientation $X^d_+$ such that for every pair of lower adjacent simplexes $\sigma, \sigma' \in X^d$, the oriented simplexes $\sigma_+, \sigma'_+ \in X^d_+$ are similarly oriented.

\subsection{Chain and Cochain Complexes}
Given a simplicial complex $X$, we can define the chain and cochain complexes of $X$ over $\R$.  The space of $j$-chains $C_j:=C_j(X;\R)$ is the vector space of linear combinations of oriented $j$-simplexes with coefficients in $\R$, with the stipulation that the two orientations of a simplex are negatives of each other in $C_j$ (as implied by our notation).  Thus, any choice of orientation $X^j_+$ provides a basis for $C_j$.  The space of $j$-cochains $C^j:=C^j(X;\R)$ is then defined to be the vector space dual to $C_j$.  These spaces are isomorphic and we will make no distinction between them.  Usually, we will work with cochains using the basis elements $\set{\textbf{1}_{\sigma_+}: \sigma_+ \in X^j_+}$, where $\textbf{1}_{\sigma_+} : C_j \to \R$ is defined on a basis element $\tau_+ \in X^j_+$ as 
\[ \textbf{1}_{\sigma_+}(\tau_+) = \begin{cases}
                                       1 & \tau_+ = \sigma_+ \\
                                       0 & \text{else}
                                      \end{cases}. \]

The boundary map $\partial_j:C_j \to C_{j-1}$ is the linear map defined on a basis element $[v_0, \ldots, v_j]$ as
\[ \partial_j[v_0, \ldots, v_j] = \sum_{i=0}^j (-1)^i[v_0, \ldots, v_{i-1}, v_{i+1}, \ldots, v_j] \]
The coboundary map $\delta^{j-1}:C^{j-1} \to C^j$ is then defined to be the transpose of the boundary map.  In particular, for $f \in C^{j-1}$,
\[ (\delta^{j-1}f)([v_0, \ldots, v_j]) = \sum_{i=1}^j (-1)^i f([v_0, \ldots, v_{i-1}, v_{i+1}, \ldots, v_j]). \]
When there is no confusion, we will denote the boundary and coboundary maps by $\partial$ and $\delta$.  It holds that $\partial \partial = \delta \delta = 0$, so that $(C_j, \partial_j)$ and $(C^j, \delta^j)$ form chain and cochain complexes.  


The homology and cohomology vector spaces of $X$ over $\R$ are 
\[ H_j := H_j(X;\R) = \frac{\ker \partial_j}{\im \partial_{j+1}} \text{\quad and \quad} H^j := H^j(X;\R) = \frac{\ker \delta^j}{\im \delta^{j-1}}. \]
It is known from the universal coefficient theorem that $H^j$ is the vector space dual to $H_j$.  Reduced (co)homology can also be used, and it is equivalent to including the nullset as a $(-1)$-dimensional simplex in $X$.

\subsection{The Laplacian}
The $k$-Laplacian of $X$ is defined to be 
\[ L_k := L_k^\text{up} + L_k^\text{down} \]
where 
\[L_k^\text{up} = \partial_{k+1}\delta^k \text{\quad and \quad} L_k^\text{down} = \delta^{k-1}\partial_k.\]
The Laplacian is a symmetric positive semi-definite matrix, as is each part $L_k^\text{up}$ and $L_k^\text{down}$.  From Hodge theory, it is known that 
\[ \ker L_k \cong H^k \cong H_k \]
and the space of cochains decomposes as
\[ C^k = \im \partial_{k+1} \oplus \ker L_k \oplus \im \delta^{k-1} \]
where the orthogonal direct sum $\oplus$ is under the ``usual'' inner product
\[ \langle f, g \rangle = \sum_{\sigma_+ \in X^k_+} f(\sigma_+)g(\sigma_+). \]
We are interested in the $L_j^{\text{down}}$ half of the Laplacian.  Trivially, $\im \partial_{j+1} \subseteq \ker L_j^\text{down}$.  The smallest nontrivial eigenvalue of $L_k^\text{down}$ is therefore given by
\[ \lambda_k  = \min_{\substack{f \in C^k \\ f \perp \im \partial}} \frac{\norm{\partial f}_2^2}{\norm{f}_2^2}, \]
where $\norm{f}_2 := \sqrt{\langle f, f \rangle}$ denotes the Euclidean norm on $C^k$.  A cochain $f$ that achieves the minimum is an eigenvector of $\lambda_k$.  It is easy to see that any such $f$ is also an eigenvector of $L_k$ with eigenvalue $\lambda_k$ and that, therefore, $\lambda_k$ relates to homology:
\[ \lambda_k = 0 \Leftrightarrow \ker L_k \neq 0 \Leftrightarrow H^k \neq 0. \] 

\begin{remark}\label{rem: Main}
Given a choice of orientation $X^k_+$, $L_k^\text{down}$ can be written as a matrix with rows and columns indexed by $X^k_+$, the entries of which are given by
\[ (L_k^\text{down})_{\sigma'_+, \sigma_+} = \begin{cases}
                                              k+1 & \sigma'_+ = \sigma_+ \\
                                              1 & \sigma'_- \sim \sigma_+ \\
                                              -1 & \sigma'_+ \sim \sigma_+ \\
                                              0 & \text{else}
                                             \end{cases}. \]
                                             
Changing the choice of orientation $X^k_+$ amounts to a change of basis for $L_k^\text{down}$.  If the row and column indexed by $\sigma_+$ are instead indexed by $\sigma_-$, all the entries in them switch sign except the diagonal entry.  Alternatively, $L_k^\text{down}$ can be characterized by how it acts on cochains:
\[ L_k^\text{down} f(\tau_+) = (k+1)\cdot f(\tau_+) + \sum_{\sigma_- \sim \tau_+} f(\sigma_+) - \sum_{\sigma_+ \sim \tau_+} f(\sigma_+). \]
Note that since $L_k^\text{down} f$ is a cochain, $L_k^\text{down} f(\tau_-) = -L_k^\text{down} f(\tau_+)$.
\end{remark}

The behavior of $L_k^\text{down}$ is related to the following concepts:
\begin{definition}
A $d$-complex $X$ is called $k$-connected ($1 \leq k \leq d$) if for every two $k$-simplexes $\sigma, \sigma'$ there exists a chain $\sigma = \sigma_0, \sigma_1, \ldots, \sigma_n=\sigma'$ of $k$-simplexes such that $\sigma_i$ is lower adjacent to $\sigma_{i+1}$ for all $i$.  For a general $d$-complex $X$, such chains define equivalence classes of $k$-simplexes, and the subcomplexes induced by these are called the $k$-connected components of $X$.
\end{definition}

\begin{definition}
A $d$-complex $X$ is called disorientable if there is a choice of orientation $X^d_+$ of its $d$-simplexes such that all lower adjacent $d$-simplexes are dissimilarly oriented.  In this case, the $d$-cochain $f = \sum_{\sigma_+ \in X^d_+} \textbf{1}_{\sigma_+}$ is called a disorientation.  
\end{definition}

\begin{remark}
Disorientability was defined in \cite{parzanchevski2012simplicial} and shown to be a higher-dimensional analogue of bipartiteness for graphs.  Note that one can also define $X$ to be $k$-disorientable if the $k$-skeleton of $X$ (the $k$-complex given by the union $\bigcup_{i \leq k}X^i$) is disorientable, but this can only happen when $k = d$.  This is not hard to see: if $k < d$ then there exists a $(k+1)$-simplex $\sigma_+ = [v_0, \ldots, v_k]$.  Given any two dissimilarly oriented faces of $\sigma_+$, say, $[v_1, v_2, \ldots, v_k]$ and $[v_0, v_2, \ldots, v_k]$, we find that the simplex $\set{v_0, v_1, v_3, \ldots, v_k}$ cannot be dissimilarly oriented to both of them simultaneously.
\end{remark}

\begin{lemma}\label{lem: spec}
 Let $X$ be a $d$-complex, $1 \leq k \leq d$ and $M = \max_{\sigma \in X^{k-1}} \deg(\sigma)$.  
\begin{enumerate}
 \item $\textup{Spec}(L_k^\text{down})$ is the disjoint union of $\textup{Spec}(L_k^\text{down}\vert_{X_i})$ where $X_i$ are the $k$-connected components of $X$.
 \item The spectrum of $L_k^\text{down}$ is contained in $[0,(k+1)M]$.
 \item The kernal of $L_k^\text{down}$ is exactly $\ker \partial_k = \im \partial_{k+1} \oplus \ker L_k$.
 \item The upper bound $(k+1)M$ is attained if and only if $k=d$ and $X$ has a $d$-connected component that is both disorientable and of constant $(d-1)$-degree.
\end{enumerate}
\end{lemma}
\begin{proof}
Statement (1) follows from the fact that $L_k^\text{down}$ can be written as a block diagonal matrix with each block corresponding to a component $X_i$.  Statement (3) is easy to verify.

For statement (2), let $f$ be an eigenvector of $L_k^\text{down}$ with eigenvalue $\lambda$, let $X^k_+$ be a choice of orientation such that $f(\sigma_+) \geq 0$ for all $\sigma_+ \in X^k_+$ and suppose $f(\tau_+) = \max_{\sigma_+ \in X^k_+} f(\sigma_+)$.  Then by Remark \ref{rem: Main},
\begin{align*}
 \lambda f(\tau_+) &= L_k^\text{down} f \\
                   &= (k+1)\cdot f(\tau_+) + \sum_{\sigma_- \sim \tau_+} f(\sigma_+) - \sum_{\sigma_+ \sim \tau_+} f(\sigma_+) \\
                        &\leq (k+1)\cdot f(\tau_+) + \sum_{\sigma_- \sim \tau_+} f(\sigma_+) + \sum_{\sigma_+ \sim \tau_+} f(\sigma_+) \\
                        &\leq (k+1)\cdot f(\tau_+) + \sum_{\sigma_- \sim \tau_+} f(\tau_+) + \sum_{\sigma_+ \sim \tau_+} f(\tau_+) \\
                        &\leq (k+1)\cdot f(\tau_+) + (k+1)(M-1)\cdot f(\sigma_+) \\
                        &\leq (k+1)M \cdot f(\tau_+)
\end{align*}
where the third inequality results from the fact that any $k$-simplex is lower adjacent to at most $(k+1)(M-1)$ other $k$-simplexes.  Therefore, $\lambda \leq (k+1)M$.  

It now remains to prove statement (4).  Looking back at the inequalities, it holds that $\lambda = (k+1)M$ only if $\sigma_- \sim \tau_+$ and $f(\sigma_+) = f(\tau_+)$ whenever $\sigma$ and $\tau$ are lower adjacent, and the faces of $\sigma$ all have degree $M$.  But since $f(\sigma_+) = f(\tau_+)$, the same reasoning can be applied to $f(\sigma_+)$ for all $\sigma$ lower adjacent to $\tau$ and eventually to all $k$-simplexes in the same $k$-connected component $X_i$.  Ultimately, this implies that $X_i$ has constant $(k-1)$-degree and is $k$-disorientable (and hence $k = d$).

To see that this bound is indeed attainable, consider a disorientable $d$-complex with constant $(d-1)$-degree $M$ (this includes, for instance, the simplicial complex induced by a single $d$-simplex).  Let $X_+^d$ be a choice of orientation such that all lower adjacent $d$-simplexes are dissimilarly oriented.  Then a disorientation $f$ on $X^d$ will satisfy
\begin{align*}
L_k^\text{down} f(\tau_+) &= (k+1)\cdot f(\tau_+) + \sum_{\sigma_- \sim \tau_+} f(\sigma_+) - \sum_{\sigma_+ \sim \tau_+} f(\sigma_+) \\
                        &=(k+1)\cdot f(\tau_+) + \sum_{\sigma_- \sim \tau_+} f(\sigma_+) \\
                        &= (k+1)\cdot 1 + \sum_{\sigma_- \sim \tau_+} 1 \\
                        &= (k+1)M \cdot 1 = (k+1)M \cdot f(\tau_+)
\end{align*}
for every $\tau_+$.
\end{proof}

\section{Random walks and the $k$-Laplacian}
\label{sec: Main}

In this section we define the $p$-lazy Dirichlet $k$-walk on $X$ and relate this walk to the spectrum of the $k$-Laplacian. \\

\paragraph{\textbf{Random walks and $L_k^\text{down}$}} 

Let $X$ be a $d$-complex, $1 \leq k \leq d$, $0 \leq p < 1$, and $M = \max_{\sigma \in X^{k-1}} \deg(\sigma)$.  
\begin{definition}
The $p$-lazy Dirichlet $k$-walk on $X$ is an absorbing Markov chain on the state space $S=X^k_\pm \cup \set{\Theta}$ defined as follows:
\begin{itemize}
\item Let two oriented $k$-cells $s, s' \in X^k_{\pm}$ be called textit{neighbors} (denoted $s \sim s'$) if they share a face and are similarly oriented.  In what follows, $\Theta$ will be used to represent an additional absorbing state, called the ``death state'', that the Markov chain can occupy.
\item Starting at an initial oriented $k$-simplex $\tau_+ \in X_{\pm}^k$, the walk proceeds as a time-homogenous Markov chain on the state space $S=X_{\pm}^k \cup \set{\Theta}$ with transition probabilities
\[ \text{Prob}(\sigma_+ \to \sigma'_+) = \text{Prob}(\sigma_- \to \sigma'_-) = \begin{cases}
p & \sigma'_+ = \sigma_+ \\
\frac{1-p}{(M-1)(k+1)} & \sigma'_+ \sim \sigma_+ \\
0 & \text{else},
\end{cases} \]
\[ \text{Prob}(\sigma_+ \to \sigma'_-) = \text{Prob}(\sigma_- \to \sigma'_+) = \begin{cases}
\frac{1-p}{(M-1)(k+1)} & \sigma'_- \sim \sigma_+ \\
0 & \text{else},
\end{cases} \]
\[ \text{Prob}(\sigma_+ \to \Theta) = \text{Prob}(\sigma_- \to \Theta) = 1-\sum_{\sigma'_+}\text{Prob}(\sigma_+ \to \sigma'_+), \]
\[ \text{Prob}(\Theta \to \Theta) = 1 \]
for all $\sigma, \sigma' \in X^k$.
\item This walk can be interpreted as follows.  Starting at $\tau_+$, the walk has probability $p$ of staying put and for each of the neighbors of $\tau_+$ the walk has probability $\frac{1-p}{(M-1)(k+1)}$ of jumping to that neighbor.  Note that if the number of neighbors of $\tau_+$ is less than $(M-1)(k+1)$, then the sum of these probabilities is less than 1.  In this case, we interpret the difference as the probability that the walker dies (i.e., the walker jumps to a death state from which it can never return). The same holds for  $\tau_-$.
\end{itemize}
\end{definition}

The left stochastic matrix for the Markov chain is a square matrix $P$ with rows and columns indexed by the state space $S=X^k_\pm \cup \set{\Theta}$ such that
\[ P_{s_1, s_2} = \text{Prob}(s_2 \to s_1) \]
for all $s_1, s_2 \in S$.  In stochastic processes it is more common to use the right stochastic matrix $P^T$ as the probability matrix, for us it will be more convenient to use the left stochastic matrix.  An initial distribution on the state space is a column vector $\nu$ indexed by $S$ such that all entries are non-negative and sum to 1.  The general framework in stochastic processes is to study how the marginal distribution $P^n \nu$ evolves as $n \to \infty$.  Indeed, one can view the Dirichlet $k$-walk as a Markov chain on a graph with vertex set $V = S$ and study the limiting behavior of $P^n  \nu$ within the context of graph theory.  However, this is not our goal.  Our goal is to connect the $k$-walk to the $k$-dimensional Laplacian, and hence to the $k$-dimensional topology and geometry of $X$.  

In order to connect the $k$-walk to $L_k$, we will not study the evolution of $P^n \nu$ but rather $T P^n \nu$, the image of the marginal distribution under a linear transformation $T$ defined as follows.  Given a choice of orientation $X^k_+ = \set{\sigma_+ : \sigma \in X^k}$, $T$ is defined to be the matrix with rows indexed by $X^k_+$ and columns indexed by $S$ such that 
\[ (T)_{\sigma_+, \sigma_+} = 1 \qquad \text{and} \qquad (T)_{\sigma_+, \sigma_-} = -1 \]
for all $\sigma \in X^k$, and such that all other entries are 0.  In other words, for any function $f : S \to \R$, $T f$ is the function $T f : X^k_+ \to \R$ such that
\[ T f (\sigma_+) = f(\sigma_+) - f(\sigma_-). \]
The definition of $T$ is motivated by geometry.  The geometry of simplicial complexes is characterized by the space of $k$-cochains $C^k$ in which $\sigma_+ = -\sigma_-$ (and for which $X^k_+$ is a choice of basis).  Probabilistically, $\sigma_+$ and $\sigma_-$ are completely separate states for the Markov chain, but geometrically we must think of them as opposite orientations of the same underlying object $\sigma$.  In addition, the state $\Theta$ has no corresponding object in $C^k$, so $T$ simply removes it from the system.  Of course, the vector $T P^n \nu$ does not have the property that it is always a distribution (all entries nonnegative and summing to 1), but it has the advantage that it resides in $C^k$ and can be related to $L_k$ as follows.  


\begin{definition}\label{def: Main}
The \textit{propagation matrix} $B$ of the Dirichlet $k$-walk is defined to be a square matrix indexed by $X^k_+$ with
\[ (B)_{\sigma'_+, \sigma_+} = \begin{cases} p & \sigma'_+ = \sigma_+ \\
-\frac{1-p}{(M-1)(k+1)} & \sigma'_+ \sim \sigma_+ \\
\frac{1-p}{(M-1)(k+1)} & \sigma'_- \sim \sigma_+ \\
0 & \text{else} \end{cases}. \]
\end{definition}    

\begin{proposition}\label{prop: Main} The propagation matrix $B$ is given by
\[ B = \frac{p(M-2)+1}{M-1}I - \frac{1-p}{(M-1)(k+1)} \cdot L_k^\text{down}. \]
In addition, $B$ satisfies $T P = B T$, so that
\[ T P^n \nu = B^n T \nu. \]
\end{proposition}
\begin{proof}
The first claim is straightforwardly checked using Definition \ref{def: Main} and Remark \ref{rem: Main}.  The second claim is equivalent to the equality $T P = B T$, which we will prove as follows.  If $s \in S$ and $P_s$ is the column of $P$ indexed by $s$, then the column of $T P$ indexed by $s$ is $T P_s$.  Using the definition of $T$, the following holds
\begin{align*} (T P)_{\sigma_+,s} &= T P_s (\sigma_+) \\
                                    &= P_s (\sigma_+) - P_s (\sigma_-) \\
                                    &= (P)_{\sigma_+,s} - (P)_{\sigma_-,s} \\
                                    &= \begin{cases} \pm p & s = \sigma_{\pm} \\ \pm \frac{1-p}{(M-1)(k+1)} & s \neq \Theta \text{ and } s \sim \sigma_{\pm} \\ 0 & \text{else} \end{cases}. 
\end{align*}
Similarly, note that $(B T)_{\sigma_+,s} = B (T \textbf{1}_s) (\sigma_+)$ where $\textbf{1}_s$ is the vector assigning 1 to $s \in S$ and 0 to all other elements in $S$.  If $s = \Theta$, $T \textbf{1}_s$ is the zero vector.  Otherwise, if $s = \tau_{\pm}$ then $T \textbf{1}_s = \pm \textbf{1}_{\tau_+}$ and
\begin{align*} (B T)_{\sigma_+,s} &= \pm B \,\textbf{1}_{\tau_+} (\sigma_+) \\
                                    &= \pm(B)_{\sigma_+,\tau_+} \\
                                    &= \begin{cases} \pm p & \tau_+ = \sigma_+ \\ \pm \frac{1-p}{(M-1)(k+1)} & \tau_+ \sim \sigma_+ \\ \mp \frac{1-p}{(M-1)(k+1)} & \tau_- \sim \sigma_+ \\ 0 & \text{else} \end{cases} \\
                                    &= \begin{cases} \pm p & s = \sigma_{\pm} \\ \pm \frac{1-p}{(M-1)(k+1)} & s \sim \sigma_{\pm} \\ 0 & \text{else} \end{cases}. 
\end{align*}                                    
This concludes the proof.
\end{proof}


For what follows, we define $\mathcal{E}_n^{\tau_+} := B^n \textbf{1}_{\tau_+}$ to be the marginal difference of the $p$-lazy Dirichlet $k$-walk on $X$ starting at $\tau_+$.  Also, let $X^k_+$ be a choice of orientation and denote $M = \max_{\sigma \in X^{k-1}} \deg(\sigma)$.  

\begin{corollary}\label{cor: B} \text{}\\
\begin{enumerate}
 \item The spectrum of $B$ is contained in $\left[2p-1, \frac{p(M-2)+1}{M-1}\right]$, with the upper bound acheived by cochains in $\ker \partial_k$ and the lower bound acheived if and only if $k=d$ and there is a disorientable $d$-connected component of constant $(d-1)$-degree.
 \item If $\tau$ has a coface, then
\[ \norm{\mathcal{E}_n^{\tau_+}}_2 \geq \left(\frac{p(M-2)+1}{M-1}\right)^n\frac{1}{\sqrt{k+2}}. \] 
 \item If $p \neq 0,1$ then
\[ \norm{\mathcal{E}_n^{\tau_+}}_2 \leq \max\left\{\abs{2p-1}^n, \left(\frac{p(M-2)+1}{M-1}\right)^n\right\}. \]
\end{enumerate}
\end{corollary}
\begin{proof}
Statement (1) is easy to verify with the help of Lemma \ref{lem: spec} and Proposition \ref{prop: Main}.  Statement (3) follows from the inequality $\norm{Af}_2 \leq \norm{A} \norm{f}_2$ where $A$ is a matrix, $f$ is a vector, and $\norm{A}$ is the spectral norm on $A$. 

It remains now to prove statement (2).  If $\tau$ has a coface $\sigma$, let $f = \partial_{k+1} \textbf{1}_{\sigma_+}$ (with $\sigma_+$ being any orientation of $\sigma$) so that $f \in \ker \partial_k$.  Let $f, f_1, \ldots, f_i$ be an orthogonal basis for $C^k$ such that $f_1, \ldots, f_i$ are eigenvectors of $B$ with eigenvalues $\gamma_1, \ldots, \gamma_i$, and assume $\textbf{1}_{\tau_+} = \alpha f + \alpha_1 f_1 + \ldots + \alpha_i, f_i$.  Then,
\begin{align}
\norm{\mathcal{E}_n^{\tau_+}}_2 &= \norm{B^n \textbf{1}_{\tau_+}}_2 \\
                                &= \norm{\alpha B^n f + \alpha_1 B^n f_1 + \ldots + \alpha_i B^n f_i}_2 \\
                                &= \abs{\alpha} \left(\frac{p(M-2)+1}{M-1}\right)^n \norm{f}_2 + \abs{\alpha_1} \gamma_1^n \norm{f_1}_2 + \ldots + \abs{\alpha_i} \gamma_i^n \norm{f_i}_2 \\
                                &\geq \abs{\alpha} \left(\frac{p(M-2)+1}{M-1}\right)^n \norm{f}_2 \\
                                &= \left(\frac{p(M-2)+1}{M-1}\right)^n\left\vert\left\langle \frac{f}{\norm{f}_2}, \textbf{1}_{\tau_+} \right\rangle\right\vert \\
                                &= \left(\frac{p(M-2)+1}{M-1}\right)^n\frac{\abs{f(\tau_+)}}{\norm{f}_2} \\
                                &= \left(\frac{p(M-2)+1}{M-1}\right)^n\frac{1}{\sqrt{k+2}}
\end{align}
\end{proof}

Note that if $p \neq 0,1$, then $\abs{2p-1}$ and $\frac{p(M-2)+1}{M-1}$ are both less than one.  Hence, the above corollary says that the limit of the marginal difference is trivial in general.  We can remove this trivial behavior by making one final alteration to our object of study: multiply the propagation matrix $B$ by $\frac{M-1}{p(M-2)+1}$ to obtain the normalized propagation matrix $\widetilde{B} := \frac{M-1}{p(M-2)+1}B$ and define $\widetilde{\mathcal{E}}_n^{\tau_+} := \widetilde{B}^n \textbf{1}_{\tau_+}$ to be the normalized marginal difference.  The next two theorems show that the homology of $X$ can be determined from the limiting behavior of the normalized marginal difference.  

\begin{theorem}\label{thm: Main} \text{}\\
The limit $\widetilde{\mathcal{E}}_\infty^{\tau_+} := \lim_{n\to \infty} \widetilde{\mathcal{E}}_n^{\tau_+}$ of the normalized marginal difference exists for all $\tau_+$ if and only if $\widetilde{B}$ has no eigenvalue $\lambda \leq -1$.  Furthermore, $\widetilde{\mathcal{E}}_\infty^{\tau_+} = \prj_{\ker \partial_k} \textbf{1}_{\tau_+}$ whenever $\widetilde{\mathcal{E}}_\infty^{\tau_+}$ exists, where $\prj_{\ker \partial_k}$ is the projection map onto $\ker \partial_k$. 
\end{theorem}
\begin{proof}
Note that by Corollary \ref{cor: B}, the spectrum of $\widetilde{B}$ is upper bounded by 1 and the eigenspace of the eigenvalue 1 is exactly $\ker \partial_k$.  Let $f_1, \ldots, f_i$ be an orthogonal basis for $C^k$ such that $f_1, \ldots, f_i$ are eigenvectors of $\widetilde{B}$ with eigenvalues $\gamma_1, \ldots, \gamma_i$.  Then any $\textbf{1}_{\tau_+}$ can be written as a linear combination $\textbf{1}_{\tau_+} = \alpha_1 f_1 + \ldots + \alpha_i, f_i$ so that
\[ \widetilde{\mathcal{E}}_\infty^{\tau_+} = \widetilde{B}^n \textbf{1}_{\tau_+} = \alpha_1 \gamma_1^n f_1 + \ldots, \alpha_i \gamma_i^n f_i \]
Since the $f_j$ form a basis, $\widetilde{\mathcal{E}}_\infty^{\tau_+}$ converges if and only if $\alpha_j \gamma_j^n$ converges for each $j$.  In other words, $\widetilde{\mathcal{E}}_\infty^{\tau_+}$ converges if and only if for every $j$, $\alpha_j = 0$ or $\gamma_j > -1$.  Furthermore, the limit (when it exists) is always
\[ \sum_{\set{j : \gamma_j = 1}} \alpha_j f_j = \prj_{\ker \partial_k} \textbf{1}_{\tau_+} \]

Finally, suppose $\widetilde{B}$ has an eigenvalue $\lambda \leq -1$.  Then there is an eigenvector $f$ such that $\widetilde{B}^n f = \lambda^n f$ does not converge.  Since the set of cochains $\set{\textbf{1}_{\tau_+} : \tau_+\in X^k_\pm}$ spans $C^k(\R)$, $f$ can be written as a linear combination of them and therefore $\widetilde{B}^n \textbf{1}_{\tau_+}$ must not converge for some $\tau_+$.
\end{proof}

\begin{theorem}\label{thm: Second} \text{}\\
\begin{enumerate}
 \item If $\frac{M-2}{3M-4}<p<1$ then the limit $\widetilde{\mathcal{E}}_\infty^{\tau_+}$ exists for all $\tau_+$ and 
\[ \dim(\textup{span}\set{\prj_{\ker \delta^k} \widetilde{\mathcal{E}}_\infty^{\tau_+} : \tau_+ \in X_{\pm}^k}) = \dim(H_k(X)) \]
where $\prj_{\ker \delta^k}$ denotes the projection map onto $\ker \delta^k$. 
\item The same holds when $p = \frac{M-2}{3M-4}$ and either $k < d$ or there are no disorientatable $d$-connected components of constant $(d-1)$-degree.   
\item We can say more if $p \geq \frac{1}{2}$.  In this case,
\[ \norm{\widetilde{\mathcal{E}}_n^{\tau_+} - \widetilde{\mathcal{E}}_\infty^{\tau_+}}_2 = O\left(\left[1-\frac{1-p}{(p(M-2)+1)(k+1)}\lambda_k\right]^n\right) \]
\end{enumerate}
\end{theorem}
\begin{proof}
The proof follows mostly from Theorem \ref{thm: Main}.  According to that theorem, $\widetilde{\mathcal{E}}_\infty^{\tau_+}$ exists for all $\tau_+$ if and only if the spectrum of $\widetilde{B}$ is contained in $(-1,1]$.  Using Corollary \ref{cor: B} and the definition $\widetilde{B} := \frac{M-1}{p(M-2)+1}B$, we know that the spectrum of $\widetilde{B}$ is contained in $\left[(2p-1)\frac{M-1}{p(M-2)+1},1\right]$.  Now,
\begin{eqnarray*}
&(2p-1)\frac{M-1}{p(M-2)+1} > -1& \\
                           &\Updownarrow &\\
&\frac{p(M-2)+1}{M-1} > 1-2p& \\
     &\Updownarrow& \\
&p\left( \frac{M-2}{M-1} + 2 \right) > 1-\frac{1}{M}& \\
              &\Updownarrow &\\
&p > \frac{M-2}{3M-4},&
\end{eqnarray*}
which proves that the spectrum of $\widetilde{B}$ is indeed contained in $(-1,1]$ when $p > \frac{M-2}{3M-4}$.  Since the $\textbf{1}_{\tau_+}$ span all of $C^k$, the $\widetilde{\mathcal{E}}_\infty^{\tau_+} = \prj_{\ker \partial_k} \textbf{1}_{\tau_+}$ span all of $\ker \partial_k$, and hence the $\prj_{\ker \delta^k} \widetilde{\mathcal{E}}_\infty^{\tau_+}$ span all of $\ker L_k$.

In the case that $p = \frac{M-2}{3M-4}$, the spectrum of $\widetilde{B}$ is contained in $[-1,1]$.  However, as long as $-1$ is not actually an eigenvalue of $\widetilde{B}$, the result still holds.  According to Corollary \ref{cor: B}, $-1$ is an eigenvalue  if and only if $k=d$ and there is a disorientable $d$-connected component of constant $(d-1)$-degree.  The case $p = 1$ is trivial ($\widetilde{B} = I$) and not considered.  

Finally, if the spectrum of $B$ lies in $(-1,1]$ and $\lambda$ is the eigenvalue of $\widetilde{B}$ contained in $(-1,1)$ with largest absolute value, so
\[\norm{\widetilde{B}^n f - \lim_{n \to \infty}\widetilde{B}^n f}_2 \leq \abs{\lambda}^n \norm{f}_2 \]
for all $f$.  Let $f_1, \ldots, f_i$ be an orthonormal basis for $C^k$ such that $f_1, \ldots, f_i$ are eigenvectors of $\widetilde{B}$ with eigenvalues $\gamma_1, \ldots, \gamma_i$.  Then any $f$ can be written as a linear combination $f = \alpha_1 f_1 + \ldots, + \alpha_i f_i$ and so that $\norm{f}_2 = \sum_j \abs{\alpha_j}$ and
\begin{eqnarray*}
\norm{\widetilde{B}^n f - \lim_{n \to \infty}\widetilde{B}^n f}_2 &=& \norm{\alpha_1 \gamma_1^n f_1 + \ldots + \alpha_i \gamma_i^n f_i - \sum_{\set{j : \gamma_j = 1}} \alpha_j f_j}_2 \\
  &=& \norm{\sum_{\set{j : \gamma_j \neq 1}} \alpha_j \gamma_j^n f_j}_2 \\
  &= &\sum_{\set{j : \gamma_j \neq 1}} \abs{\alpha_j \gamma_j^n} \norm{f_j}_2 \\
  &\leq& \sum_{\set{j : \gamma_j \neq 1}} \abs{\alpha_j} \abs{\lambda}^n \\
  &\leq &\abs{\lambda}^n \norm{f}_2
\end{eqnarray*}
In particular, if $p \geq \frac{1}{2}$ then the spectrum of $\widetilde{B}$ is contained in $[0,1]$ and therefore $\lambda = 1-\frac{1-p}{(p(M-2)+1)(k+1)}\lambda_k$.
\end{proof}

Note the dependence of the theorem on both the lazy probability $p$ and on $M$.  We can think of $M$ as the maximum amount of ``branching'', where $M=2$ means there is no branching, as in a pseudomanifold of dimension $d=k$, and large values of $M$ imply a high amount of branching.  In particular, the walk must become more and more lazy for larger values of $M$ in order to prevent the marginal difference from diverging.  However, since $\frac{M-2}{3M-4} < \frac{1}{3}$ for all $M$ a lazy probability of at least $\frac{1}{3}$ will always ensure convergence.  While there is no explicit dependence on $k$ or the dimension $d$, it is easy to see that $M$ must always be at least $d - k + 1$ (for instance, it is not possible for a triangle complex to have maximum vertex degree 1).

We would also like to know whether for the normalized marginal difference converges to 0.  Note that if $\tau_+$ has a coface, then we already know that $\norm{\mathcal{E}_n^{\tau_+}}_2$ stays bounded away from 0 according to Corollary \ref{cor: B}.  However, if $\tau$ has no coface, then $\textbf{1}_{\tau_+}$ may be perpendicular to $\ker \partial_k$, allowing $\norm{\mathcal{E}_n^{\tau_+}}_2$ to die in the limit as we see in the following corollary.

\begin{corollary}
\label{cor: Main}
If $\tau$ has no coface, $H_k = 0$, and if $\frac{M-2}{3M-4}<p<1$ then
\[ \norm{\mathcal{E}_\infty^{\tau_+}}_2 = 0. \]
The same is true when $p = \frac{M-2}{3M-4}$ and either $k < d$ or there are no disorientable $d$-connected components of constant $(d-1)$-degree,
\end{corollary}
\begin{proof}
Under all conditions stated, $\widetilde{\mathcal{E}}_\infty^{\tau_+}$ converges.  If $\tau$ has no coface, then $\textbf{1}_{\tau_+}$ is in the orthogonal complement of $\im \partial_{k+1}$, because all elements of $\im \partial_{k+1}$ are supported on oriented faces of $(k+1)$-simplexes.  If $H_k = 0$ then $\ker \partial_k = \im \partial_{k+1}$, so that 
\[ \norm{\widetilde{\mathcal{E}}_\infty^{\tau_+}}_2 = \prj_{\ker \partial_k} \textbf{1}_{\tau_+} = 0. \]
\end{proof}

\section{Random walks with Neumann boundary conditions}
The Neumann random walk described by Rosenthal and Parzanchevski in \cite{parzanchevski2012simplicial} is the ``dual'' of the Dirichlet random walk, jumping from simplex to simplex through cofaces rather than faces.  Let $X$ be a $d$-complex, $0 \leq k \leq d-1$, and $0 \leq p < 1$.  
\begin{definition} The $p$-lazy Neumann $k$-walk on $X$ is an absorbing markov chain on the state space $S = X^k_{\pm} \cup \set{\Theta}$ defined as follows:\\
\begin{itemize}
\item Let two oriented $k$-simplexes $s, s' \in X_{\pm}^k$ be called \textit{coneighbors} (denoted $s \frown s'$) if they share a coface and are dissimilarly oriented.  Also, let $\deg(\sigma)$ denote the number of cofaces of $\sigma$.  In what follows, $\Theta$ is an additional absorbing state the random walk can occupy, called the ``death state''.
\item Starting at an initial oriented $k$-simplex $\tau_+ \in X_{\pm}^k$ the walk proceeds with as a time-homogeneous Markov chain on 
$S := X^k_{\pm} \cup \set{\Theta}$ with transition probabilities 
\[ \text{Prob}(\sigma_+ \to \sigma'_+) = \text{Prob}(\sigma_- \to \sigma'_-) = \begin{cases}
p & \sigma'_+ = \sigma_+ \\
\frac{1-p}{k \cdot \deg(\sigma)} & \sigma'_+ \frown \sigma_+ \\
0 & \text{else},
\end{cases} \]
\[ \text{Prob}(\sigma_+ \to \sigma'_-) = \text{Prob}(\sigma_- \to \sigma'_+) = \begin{cases}
\frac{1-p}{k \cdot \deg(\sigma)} & \sigma'_- \frown \sigma_+ \\
0 & \text{else},
\end{cases} \]
\[ \text{Prob}(\sigma_+ \to \Theta) = \text{Prob}(\sigma_- \to \Theta) = \begin{cases} 1-p & \deg(\sigma) = 0 \\ 0 & \text{else} \end{cases}, \]
\[ \text{Prob}(\Theta \to \Theta) = 1. \]
for all $\sigma, \sigma' \in X^k$.

\item This walk can be described as follows.  Starting at at any $\sigma_+$, the walk has a probability $p$ of staying put and otherwise is equally likely to jump to one of the $k \cdot \deg(\sigma)$ coneighbors of $\sigma_+$.  If $\sigma$ has no coneighbors (i.e., if $\sigma$ has no cofaces), then the walk instead has probability $p$ of staying put and probability $1-p$ of jumping to the absorbing state $\Theta$. The same holds for starting at $\sigma_-$.
\end{itemize}
\end{definition}

This definition varies from that in \cite{parzanchevski2012simplicial} where the case of $k = d-1$ was examined and it was assumed that every $k$-simplex had at least one coface, and as a result a death state was not required.  The inclusion of the death state in all cases in the definition above allows us to use the matrix $T$ from Section \ref{sec: Main} to relate the marginal distribution of the walk to $L_k^\text{up}$.  If $\nu$ is an initial distribution and $P$ is the left stochastic matrix for the walk (so that $P^n \nu$ is the marginal distribution after $n$ steps), then $T P^n \nu$ is the marginal difference after $n$ steps for the Neumann $k$-walk.  Similar to the Dirichlet walk, there is a propagation matrix $A$ such that $T P^n \nu = A^n T \nu$ and such that $A$ relates to $L_k^\text{up}$.  Once again the marginal difference converges to 0 for all initial distributions, but this behavior is fixed by multiplying $A$ by a constant, obtaining a normalized propagation matrix $\widetilde{A}$ and a normalized marginal distribution $\widetilde{A}^n T \nu$.  The limiting behavior of the normalized marginal difference reveals homology similar to Theorem \ref{thm: Second}.  

While the results for the Neumann and Dirichlet walks are quite similar, we highlight two differences. One is that the norm of the normalized marginal difference for the Neumann $k$-walk starting at a single oriented simplex stays bounded away from 0 (see Proposition 2.8 of \cite{parzanchevski2012simplicial}), whereas this need not hold for the Dirichlet $k$-walk (as in Corollary \ref{cor: Main}).  This is because in the Neumann case, every starting point $\textbf{1}_{\tau_+}$ has some nonzero inner product with an element of $\im \delta^{k-1} \subseteq \ker \delta^k$.  The second difference is in the threshold values for $p$ in Theorem \ref{thm: Second} and in the corresponding Theorem 2.9 of \cite{parzanchevski2012simplicial}.  For the Dirichlet walk, homology can be detected for $p > \frac{M-2}{3M-4}$ (where $M = \max_{\sigma \in X^{k-1}} \deg(\sigma)$) whereas for the Neumann walk the threshold is $p > \frac{k}{3k+2}$.  Hence, the Neumann walk is sensitive to the dimension while the Dirichlet walk is sensitive to the maximum degree.  In both cases, $p \geq \frac{1}{3}$ is always sufficient to detect homology and $p \geq \frac{1}{2}$ allows us to put a bound on the rate of convergence.



\section{Other  Random Walks}
\label{sec: other}

The examples of the Dirichlet random walk and the Neumann random walk suggest that a more general method for relating matrices to random walks is possible.  So far only the unweighted Laplacian matrices $L_k^\text{up}$ and $L_k^\text{down}$ have been found to relate to random walks, but one might ask whether the full Laplacian matrix $L_k = L_k^\text{up} + L_k^\text{down}$ as well as weighted Laplacians can be related to random walks.  Weighted Laplacians will not be considered in this paper, but can be defined as
\[\mathcal{L}_k = \mathcal{L}_k^\text{up} + \mathcal{L}_k^\text{down} \]
where
\[ \mathcal{L}_k^\text{up} := W_k^{-1/2} \partial_{k+1} W_{k+1} \delta^k W_k^{-1/2} \text{ and } \mathcal{L}_k^\text{down} := W_k^{1/2} \delta^{k-1} W_{k-1}^{-1} \partial_k W_k^{1/2} \]
and where $W_j$ denotes a diagonal matrix with diagonal entries equal to positive weights, one for each $j$-simplex.  In order to make a broad theorem relating Laplacians to random walks, we introduce the following notion of an ``$X^k_+$-matrix''.

\begin{definition}\label{def: matrix}
Let $X^k_+$ be a choice of orientation.  An $X^k_+$-matrix is a square matrix $L$ such that
\begin{enumerate}
\item the rows and columns of $L$ are indexed by $X^k_+$,
\item $L$ has nonnegative diagonal entries,
\item whenever $L$ has a zero on the diagonal, all other entries in the same row or column are also zero.
\end{enumerate}
\end{definition}

\begin{definition}\label{def: props}
Let $X^k_+$ be a choice of orientation, $L$ an $X^k_+$-matrix, and $p \in [0,1]$.  We define the $p$-lazy propagation matrix related to $L$ to be
\[ A_{L,p} := \frac{p(K-1)+1}{K}I - \frac{1-p}{K} \cdot L D_L^{-1} \]
where $p \in [0,1]$, $K := \max_{\sigma_+ \in X^k_+} \sum_{\sigma'_+ \neq \sigma_+} \abs{(L D_L^{-1})_{\sigma'_+, \sigma_+}}$, and $D_L$ is the diagonal matrix with the same nonzero diagonal entries as $L$ and with all other diagonal entries equal to 1 (or any nonzero number, as property (3) of Definition \ref{def: matrix} ensures $L D_L^{-1}$ will be unchanged).  The case $K=0$ is degenerate and not considered.  If $(D_L)_{\sigma_+, \sigma_+} = 0$, then $(D_L^{-1})_{\sigma_+, \sigma_+} = 0$ by convention.  In addition, we define the normalized $p$-lazy propagation matrix relating to $L$ to be
\[ \widetilde{A}_{L,p} := I - \frac{1-p}{p(K-1)+1}L D_L^{-1} \left(= \frac{K}{p(K-1)+1}A_{L,p} \right) \]
\end{definition}

Note that whenever $K = 1$, $A_{L,p} = \widetilde{A}_{L,p}$.  In particular, this is true in the graph case when $L = L_0$.

\begin{definition}
Let $X^k_+$ be a choice of orientation, $L$ an $X^k_+$-matrix, $p \in [0,1]$, and let $A_{L,p}$ be defined as above.  We define $P_{L,p}$ to be the square matrix with rows and columns indexed by $S := X^k_+ \cup \set{\Theta}$ with
\[ (P_{L,p})_{\sigma'_+, \sigma_+} = (P_{L,p})_{\sigma'_-, \sigma_-} = \begin{cases} (A_{L,p})_{\sigma'_+, \sigma_+} & \text{if } (A_{L,p})_{\sigma'_+, \sigma_+} > 0 \\ 0 & \text{else} \end{cases}, \]
\[ (P_{L,p})_{\sigma'_-, \sigma_+} = (P_{L,p})_{\sigma'_+, \sigma_-} = \begin{cases} -(A_{L,p})_{\sigma'_+, \sigma_+} & \text{if } (A_{L,p})_{\sigma'_+, \sigma_+} < 0 \\ 0 & \text{else} \end{cases}, \]
\[ (P_{L,p})_{s, \Theta} = 0 \text{ for all } s \neq \Theta, \]
\[ (P_{L,p})_{\Theta, s} = 1 - \sum_{s' \in S \setminus \set{\Theta}} (P_{L,p})_{s', s} \text{ for all } s \neq \Theta, \]
and
\[ (P_{L,p})_{\Theta, \Theta} = 1. \]
\end{definition}

The following lemma says that $P_{L,p}$ is always a probability matrix.

\begin{lemma}
Let $X^k_+$ be a choice of orientation, $L$ an $X^k_+$-matrix, and $p \in [0,1]$.  The matrix $P_{L,p}$ defined above is the left stochastic matrix for an absorbing Markov chain on the state space $S$ (i.e., $(P_L)_{s',s} = \text{Prob}(s \to s')$) such that $\Theta$ is an absorbing state and $\text{Prob}(s \to s) = p$ for all $s \neq \Theta$.  
\end{lemma}
\begin{proof}
It is clear by the definition of $P_{L,p}$ that $\Theta$ is an absorbing state.  To see that $\text{Prob}(s \to s) = p$ for all $s \neq \Theta$, note that
\begin{align*}
(A_{L,p})_{\sigma_+, \sigma_+} &= \frac{p(K-1)+1}{K} - \frac{1-p}{K} \cdot 1 \\
                               &= \frac{p(K-1)+1-1+p}{K} = p
\end{align*}
and hence by the definition of $P_{L,p}$,
\[ (P_{L,p})_{\sigma_-, \sigma_-} = (P_{L,p})_{\sigma_+, \sigma_+} = p \]
for all $\sigma$.  It is also clear by the definition of $P_{L,p}$ that the entries $(P_{L,p})_{\sigma'_-, \sigma_+} = (P_{L,p})_{\sigma'_+, \sigma_-}$ are nonnegative for any $\sigma, \sigma'$.  Hence, in order to show that $P_{L,p}$ is left stochastic we need only to prove that $\sum_{s' \in S \setminus \set{\Theta}} (P_{L,p})_{s', s} \leq 1$ for all $s \in S \setminus \set{\Theta}$.  By the symmetries inherent in $P_{L,p}$, the value of the sum is the same for $s = \sigma_+$ as it is for $s = \sigma_-$.  For any $s = \sigma_+$,
\begin{align*}
\sum_{s' \in S \setminus \set{\Theta}} (P_{L,p})_{s', s} &= \sum_{\sigma'_+ \in X^k_+} (A_{L,p})_{\sigma'_+, \sigma_+} \\
                                                         &= p + \sum_{\sigma'_+ \in X^k_+ \setminus \set{\sigma_+}} \abs{(A_{L,p})_{\sigma'_+, \sigma_+}} \\
                                                         &= p + \frac{1-p}{K}\sum_{\sigma'_+ \in X^k_+ \setminus \set{\sigma_+}} \abs{(L D_L^{-1})_{\sigma'_+, \sigma_+}} \\
                                                         &\leq p + (1-p) = 1.
\end{align*}
This completes the proof.
\end{proof}

We will call $P_{L,p}$ the $p$-lazy probability matrix related to $L$.  The following theorem shows that $P_{L,p}$ is related $L$.

\begin{theorem}
Let $X^k_+$ be a choice of orientation, $L$ an $X^k_+$-matrix, $p \in [0,1]$, and let $A_{L,p}$ and $P_{L,p}$ be defined as above.  In addition, let $T_+$ be defined as in section \ref{sec: Main}.  Then
\[ A_{L,p} T = T P_{L,p}. \]
In other words, the evolution of the marginal differences $T_+ P_{L,p}^n \nu$ after $n$ steps with initial distribution $\nu$ is governed by the propagation matrix: $T P_{L,p}^n \nu = A_{L,p}^n T \nu$.
\end{theorem}
\begin{proof}
Using the definition of $T$
\begin{align*} (T P_{L,p})_{\sigma_+,s} &= (P_{L,p})_{\sigma_+,s} - (P_{L,p})_{\sigma_-,s} \\
                                          &= \begin{cases} \pm (A_{L,p})_{\sigma_+,\sigma'_+} & s = \sigma'_{\pm} \\ 0 & s = \Theta \end{cases}.
\end{align*}
Similarly, note that $(A_{L,p} T)_{\sigma_+,s} = A_{L,p} (T \textbf{1}_s) (\sigma_+)$ where $\textbf{1}_s$ is the vector assigning 1 to $s \in S$ and 0 to all other elements in $S$.  If $s = \Theta$, $T \textbf{1}_s$ is the zero vector.  Otherwise, if $s = \tau_{\pm}$ then $T \textbf{1}_s = \pm \textbf{1}_{\tau_+}$.  Thus,
\begin{align*} (A_{L,p} T)_{\sigma_+,s} &= \begin{cases} \pm A_{L,p} \textbf{1}_{\tau_+} (\sigma_+) & s = \tau_{\pm} \\ 0 & s = \Theta \end{cases} \\
                                          &= \begin{cases} \pm (A_{L,p})_{\sigma_+,\tau_+} & s = \tau_{\pm} \\ 0 & s = \Theta \end{cases}.
\end{align*}                                    
This concludes the proof.
\end{proof}

Finally, we conclude with a few results motivating the normalized propagation matrix and showing how the limiting behavior of the marginal difference relates to the kernel and spectrum of $L$.  We strongly suspect stronger results hold.

\begin{theorem}
Let $X^k_+$ be a choice of orientation, $L$ an $X^k_+$-matrix with $\text{Spec}(L) \subset [0, \Lambda]$ ($\Lambda > 0$).  Then for $\frac{\Lambda - 1}{K + \Lambda - 1} \leq p < 1$ the following statements hold:
\begin{enumerate}
\item $\norm{A_{L,p}^n T \nu}_2 \to 0$ for every initial distribution $\nu$,
\item $\widetilde{A}_{L,p}^n T \nu \to \prj_{\ker L} T\nu$ for every initial distribution $\nu$, where $\prj_{\ker L}$ denotes the projection map onto the kernel of $L$,
\item If $\lambda$ is the spectral gap (smallest nonzero eigenvalue) of $L$ then 
\[ \norm{\widetilde{A}_{L,p}^n T \nu - \prj_{\ker L} T \nu}_2 = O\left( \left[ 1 - \frac{1-p}{p(K-1)+1}\lambda \right]^n \right). \]
\end{enumerate} 
\end{theorem}
\begin{proof}
The proof is the same as in the proofs of Corollary \ref{cor: B} and Theorem \ref{thm: Second} and mostly boil down to statements about the spectra of $A_{L,p}$ and $\widetilde{A}_{L,p}$.  Note that since $\frac{\Lambda - 1}{K + \Lambda - 1} \leq p < 1$, $\text{Spec}(\widetilde{A}_{L,p}) \subset [0,1]$ where the eigenspace of the eigenvalue 1 is equal to the kernel of $L$, and the largest eigenvalue of $\widetilde{A}_{L,p}$ less than 1 is $1 - \frac{1-p}{p(K-1)+1}\lambda$.  
\end{proof}

As an example of the applicability of this framework, $\widetilde{A}_{L,p}$ is used with $L = L_k$ to perform label propagation on edges in the next section.

\section{Examples of random walks}

In this section we state some specific random walks to provide some intuition for random walks on complexes and to use the ideas we have developed to
study a problem in machine learning, semi-supervised learning.

\subsection{Triangle complexes}\label{sec: ex}
We begin by reviewing local random walks on graphs as defined by Fan Chung in \cite{chung2007random}.  Given a graph $G = (V,E)$ and a designated ``boundary'' subset $S \subset V$, a $\frac{1}{2}$-lazy random walk on $\overline{S} = V \setminus S$ can be defined which satisfies a Dirichlet boundary condition on $S$ (meaning a walker is killed whenever it reaches $S$).  The walker starts on a vertex $v_0 \in \overline{S}$ and at each step remains in place with probability $\frac{1}{2}$ or else jumps to one of the adjacent vertices with equal probability.  The boundary condition is enforced by declaring that whenever the walker would jump to a vertex in $S$, the walk ends.  Thus, the left stochastic matrix $P$ for this walk can be written down as
\[ (P)_{v',v \in \overline{S}} = \text{Prob}(v \to v') = \begin{cases}
                \frac{1}{2} & \text{if } v = v' \\
		\frac{1}{2 d_v} & \text{if } v \sim v' \\
		0 & \text{else}
               \end{cases}
\]
where $v \sim v'$ denotes that vertices $v$ and $v'$ are adjacent and $d_v$ is the number of edges connected to $v$.  Note that $P$ is indexed only by $\overline{S}$, and that its columns sums may be less than 1.  The probability of dying is implicitly encoded in $P$ as the difference between the column sum and 1.  As was shown in \cite{chung2007random}, $P$ is related to a local Laplace operator also indexed by $\overline{S}$.  If $D$ is the degree matrix and $A$ the adjacency matrix, the graph Laplacian of $G$ is $L = D - A$.  We denote the local Laplacian as $L_S$, where $S$ in subscript means rows and columns indexed by $S$ have been deleted.  The relation between $P$ and $L_S$ is
\[ P = I - \frac{1}{2}L_S D_S^{-1}. \]
Hence, the existence and rate of convergence to a stationary distributions can be studied in terms of the spectrum of the local Laplace operator.  

Now suppose we are given an orientable 2-dimensional non-branching simplicial complex $X = (V,E,T)$ where $T$ is the set of triangles (subsets of $V$ of size 3).  Non-branching means that every edge is contained in at most 2 triangles.  We can define a random walk on triangles fundamentally identical to a local walk on a graph which reveals the 2-dimensional homology of $X$.  The $\frac{1}{2}$-lazy Dirichlet $2$-walk on $T$ starts at a triangle $t_0$ and at each step remains in place with probability $\frac{1}{2}$ or else jumps to the other side of one of the three edges.  If no triangle lies on the other side of the edge, the walk ends.  The transition matrix $B$ for this walk is given by
\[ (B)_{t',t} = \text{Prob}(t \to t') = \begin{cases}
                 \frac{1}{2} & \text{if } t = t' \\
		 \frac{1}{6} & \text{if } t \sim t' \\
		 0 & \text{else}
                \end{cases}
\]
where $t \sim t'$ denotes $t$ and $t'$ share an edge.  This is the same transition matrix as $P$, in the case that $d_v = 3$ for all $v \in \overline{S}$.  In this case, the analog of the set $S$ is the set of edges that are contained in only one triangle, which is the boundary of $X$.  To draw an explicit connection, imagine adding a triangle to each boundary edge, obtaining a larger complex $\widetilde{X}=(\widetilde{V},\widetilde{E},\widetilde{T})$.  See Figure \ref{fig: phantom}

\begin{figure}
\centering
\resizebox{0.95\textwidth}{!}{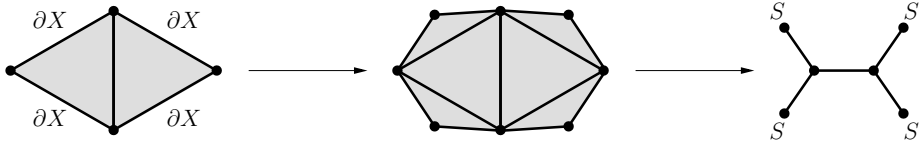}\\
\caption{Making the Dirichlet boundary condition explicit, and translating into a graph.}
\label{fig: phantom}
\end{figure}

Then take the ``dual graph'' $G=(V,E)$ of $\widetilde{X}$ by thinking of triangles as vertices (so, $V = \widetilde{T}$) and connecting vertices in $G$ with an edge if the corresponding triangles in $\widetilde{X}$ share an edge.  Choose the vertices corresponding to the added triangles $\widetilde{T}\setminus T$ to be the boundary set $S$.  Now the matrix $P$ associated to the local random walk on $G$ is indistinguishable from the matrix $B$ associated to the random walk on $X$.  In addition, it can be seen that $L_S$ on $G$ is the same as $L_2$, the 2-dimensional Laplacian on $X$ defined with respect to a given orientation  we 
have assumed orientability assumption). The following states the relation between the transition matrices and Laplacians:
\[ B = P = I - \frac{1}{6} L_S = I - \frac{1}{6}L_2. \]
See section \ref{sec: Def} for the definition of $L_2$, and the appendix of \cite{steenbergen2012cheeger} for more on the connection between 
$L_S$ and $L_2$.  

It is a basic fact that the kernel of $L_2$ corresponds to the 2-dimensional homology group of $X$ over $\R$.  Therefore, there exists a stationary distribution for the random walk if and only if $X$ has nontrivial homology in dimension 2.  Additionally, the rate of convergence to the stationary distribution (if it exists) is governed by the spectral gap of $L_2$.  In particular, the following statements hold:
\begin{enumerate}
 \item Given a starting triangle $t_0$, the marginal distribution of the random walk after $n$ steps is $\mathcal{E}_n^{t_0} := B^n \textbf{1}_{t_0}$ where $\textbf{1}_{t_0}$ is the vector assigning a 1 to $t_0$ and 0 to all other triangles.  For any $t_0$, the marginal distrubition converges, i.e., $\mathcal{E}_\infty^{t_0}:=\lim_{n \to \infty} \mathcal{E}_n^{t_0}$ exists.
 \item The limit $\mathcal{E}_\infty^{t_0}$ is equal to 0 for all starting triangles $t_0$ if and only if $X$ has trivial homology in dimension 2 over $\R$.
 \item The rate of convergence is given by
 \[ \norm{\mathcal{E}_n^{t_0} - \mathcal{E}_\infty^{t_0}}_2 = O\left(\left[1 - \frac{1}{6}\lambda_2\right]^n\right) \]
 where $\lambda_2$ is the smallest nonzero eigenvalue of $L_2$.
\end{enumerate}

The example given here is constrained by certain assumptions (orientability and the non-branching property), which allows for the most direct interpretation with respect to previous work done on graphs.

\subsection{Label propagation on edges}

In machine learning random walks on graphs have been used for semi-supervised learning.  In this section we will generalize a class of algorithms on
graphs called ``label propogation" algorithms to simplicial complexes, specifically we extend the algorithm described in \cite{zhu2005semi} (for more examples, see \cite{callut2008semi,jaakkola2002partially,zhou2004learning}). The goal of semi-supervised classification learning is to classify a set of unlabelled objects $\set{v_1, \ldots, v_u}$, given a small set of labelled objects $\set{v_{u+1}, \ldots, v_{u+\ell}}$ and a set $E$ of pairs of objects $\set{v_i,v_j}$ that one believes \textit{a priori} to share the same class.  Let $G = (V,E)$ be the graph with vertex set $V = \set{v_1, \ldots, v_{u+\ell}}$ and let $P$ be the probability matrix for the usual random walk, i.e.,
\[ (P)_{ij} = \text{Prob}(v_j \to v_i) = \frac{1}{d_j} \]
where $d_j$ is the degree of vertex $j$.  We denote the classes an object belongs to as $c=1,...,C$ and an initial distribution $f_0^c : V \to [0,1]$
is the \textit{a priori} confidence that each vertex is in class $c$, a recursive label propagation process proceeds as follows.
\begin{enumerate}
\item For $t =1,..., T$ and $c=1,..,C$:
\begin{enumerate}
\item Set $f_{t}^c  \leftarrow P f_{t-1}^c$
\item Reset $f_{t}^c(v_i) = 1$ for all $v_i$ labelled as $c$.
\end{enumerate}
\item Consider $f_T^c$ as an estimate of the relative confidence that each object is in class $c$.
\item For each unlabelled point $v_i$, $i \leq u$, assign the label 
\[\argmax_{c=1,..C} \{f_T^{c}(v_i)\}.\]
\end{enumerate}
The number of steps $T$ is set to be large enough such that $f_T^c$ is close to its limit $f_\infty^c := \lim_{T \to \infty} f_T^c$.  If $G$ is connected, it can be shown that $f_\infty^c$ is independent of the choice of $f_0^c$.  Even if $G$ is disconnected, the algorithm can be performed on each connected component separately and again the limit $f_\infty^c$ for each component will be independent of the choice of $f_0^c$.

We will now adapt the label propagation algorithm to higher dimensional walks, namely, walks on oriented edges. Given any random walk on the set of oriented edges (and an absorbing death state $\Theta$), its probability transition matrix $P$ could be used to propagate labels in the same manner as the above algorithm. However, this will treat and label the two orientations of a single edge separately as though they are unrelated. As found in this paper and in \cite{parzanchevski2012simplicial}, geometric meaning and interesting long-term behavior is obtained by transforming and normalizing $P$ into a normalized propagation matrix, and applying it not to functions on the state space but to $1$-cochains.  In this way we will infer only one label per edge. One major change, however, is that labels will become oriented themselves.  That is, given an oriented edge $e_+$ and a class $c$, the propagation algorithm may assign a positive confidence that $e_+$ belongs to class $c$ or a negative confidence that $e_+$ belongs to class $c$, which we view as a positive confidence that $e_+$ belongs to class $-c$ or, equivalently, that $e_-$ belongs to class $c$.  This construction applies to systems in which every class has two built-in orientations or signs, or the class information has a directed sense of ``flow''.

For example, imagine water flowing along a triangle complex in two dimensions.  Given an oriented edge, the water may flow in the positive or negative direction along the edge.  A ``negative'' flow of water in the direction of $e_+$ can be interpreted as a positive flow in the direction of $e_-$.  Perhaps the flow along a few edges is observed and one wishes to infer the direction of the flow along all the other edges.  Unlike in the graph case, a single class of flow already presents a classification challenge.  Or consider multiple streams of water colored according to the $C$ classes, we may want to know which stream dominates the flow along each edge and in which direction.  In order to make these inferences, it is necessary to make some assumption about how labels should propagate from one edge to the next.  When considering water flow, it is intuitive to make the following two assumptions.

\begin{enumerate}
\item \textbf{Local Consistency of Motion.}  If water is flowing along an oriented edge $[v_i,v_j]$ in the positive direction, then for every triangle $[v_i, v_j, v_k]$ the water should also tend to flow along $[v_i, v_k]$ and $[v_k, v_j]$ in the positive directions.
\item \textbf{Preservation of Mass.}  The total amount of flow into and out of each vertex (along edges connected to the vertex) should be the same.
\end{enumerate}

In fact, either one of these assumptions is sufficient to infer oriented class labels given the observed flow on a few edges.  Depending on which assumptions one chooses, different normalized propagation matrices $\widetilde{A}_{L,p}$ (see section \ref{sec: other}) may be applied. For example, $L = L_1^\text{up}$ will enforce local consistency of motion without regard to preservation of mass, while $L = L_1^\text{down}$ will do the opposite. A reasonable way of preserving both assumptions is by using $L = L_1$ as shown in Example \ref{ex: L1}.  

We now state a simple algorithm, analogous to the one for graphs, that propagates labels on edges to infer a partially-observed flow. Let $X$ be a simplicial complex of dimension $d \geq 1$ and let $X^1_+ = \set{e_1, \ldots, e_n}$ be a choice of orientation for the set of edges.  Without loss of generality, assume that oriented edges $e_u+1, \ldots, e_{n = u+\ell}$ have been classified with class $c$ (not $-c$).  Similar to the graph case, we apply a recursive label propagation process to an initial distribution vector $f_0^c : X^1_+ \to \R$ measuring the \textit{a priori} confidence that each oriented edge is in class $c$.  
See Algorithm \ref{alg1} for the procedure. The result of the algorithm is a set of estimates of the relative confidence that each edge is in class $c$ with some orientation.

\begin{algorithm}
\caption{Edge propagtion algorithm.}\label{alg1}
 \SetAlgoLined
 \KwData{Simplicial complex $X$, set of oriented edges \[X^1_+ = \{e_1, \ldots, e_u, e_{u+1},...,e_{u+\ell}\}\] with $e_{u+1}, \ldots, e_{u + \ell}$ labelled with oriented classes $\pm 1,.., \pm C$, initial distribution vector $f_0^c : X^1_+ \to \R$, number of iterations $T$}
 \KwResult{Confidence of class membership and direction for unlabelled edges $\{f_*^c(e_1) ,..., f_*^c(e_u)\}_{c=1}^C$}
 \For{$c=1$ to $C$}
	{\For{$t=1$ to $T$}
		{ $f_{t}^c \leftarrow \widetilde{A}_{L,p}f_{t-1}^c$\;
		 $f_{t}^c(e_i) \leftarrow 1$ for $e_i$ labelled with class $c$\;
		 $f_{t}^c(e_i) \leftarrow -1$ for $e_i$ labelled with class $-c$ }
 	}
 {$\{f_*^c(e_1) ,..., f_*^c(e_u)\}_{c=1}^C \leftarrow \{f_T^c(e_1) ,..., f_T^c(e_u)\}_{c=1}^C$\;}
\end{algorithm}

After running the algorithm, an unlabelled edge $e_i$ is assigned the oriented class $\text{sgn}(f_T^c(e_i))c$ where $c = \argmax_{c=1,..C} \set{\abs{f_T^{c}(e_i)}}$.

We now prove that given enough iterations $T$ the algorithm converges and the resulting assigned labels are meaningful. The proof uses the same methods as the one found in \cite{zhu2005semi} for the graph case.  

\begin{proposition}
Using the notation of section \ref{sec: other}, assume that $L$ is a symmetric $X^k_+$-matrix with $\text{Spec}(L D_L^{-1}) \subset [0, \Lambda]$.  Let $\widetilde{A}_{L,p}$ be the normalized $p$-lazy propagation matrix as defined in \ref{def: props}.  If $\frac{\Lambda - 2}{2K + \Lambda - 2} < p < 1$ and if no vector in $\ker L$ is supported on the set of unclassified edges, then Algorithm \ref{alg1} converges.  That is,
\[ \lim_{T \to \infty} f^c_T =: f_\infty^c = \begin{pmatrix} \psi^c \\ (I-A_4)^{-1}A_3 \psi^c \end{pmatrix}, \] 
where $A_4$ and $A_3$ are submatrices of $\widetilde{A}_{L,p}$ and $\psi^c$ is the class function on edges labelled with $\pm c$ (for which $\psi^c(e_i) = \pm 1$).  In addition, $f_\infty^c$ depends neither on the initial distribution $f_0^c$ nor on the lazy probability $p$.
\end{proposition}
\begin{proof}
First, note that we are only interested in the convergence of $f_T^c(e_i)$ for $e_i$ not labelled $\pm c$.  Partition $f_T^c$ and $\widetilde{A}_{L,p}$ according to whether $e_i$ is labelled $\pm c$ or not as
\[ f_T^c = \begin{pmatrix} \psi^c \\ \hat{f}_T^c \end{pmatrix} \qquad \text{and} \qquad \widetilde{A}_{L,p} = \begin{pmatrix} A_1 & A_2 \\ A_3 & A_4 \end{pmatrix}. \]
The recursive definition of $f_T^c$ in Algorithm \ref{alg1} can now be rewritten as $\hat{f}_T^c = A_4 \hat{f}_{T-1}^c + A_3 \psi^c$.  Solving for $\hat{f}_T^c$ in terms of $\hat{f}_0^c$ yields
\[ \hat{f}_T^c = (A_4)^k \hat{f}_0^c + \sum_{i=0}^{T-1} (A_4)^i A_3 \psi^c. \]
In order to prove convergence of $\hat{f}_T^c$, it suffices to prove that $A_4$ has only eigenvalues strictly less than 1 in absolute value.  This ensures that $(A_4)^k\hat{f}_0^c$ converges to zero (eliminating dependence on the initial distribution) and that $\sum_{i=0}^{k-1} (A_4)^i A_3 \psi^c$ converges to $(I-A_4)^{-1}A_3 \psi^c$ as $k \to \infty$.  We will prove that $\text{Spec}(A_4) \subset (-1,1)$ by relating $\text{Spec}(A_4)$ to $\text{Spec}(L D_L^{-1})\subset [0, \Lambda]$ as follows. 

First, partition $L$ and $D_L$ similar to $\widetilde{A}_{L,p}$ as
\[ L = \begin{pmatrix} L_1 & L_2 \\ L_3 & L_4 \end{pmatrix} \qquad \text{and} \qquad D_L = \begin{pmatrix} D_1 & 0 \\ 0 & D_4 \end{pmatrix}. \]
so that
\[ A_4 = I - \frac{1-p}{p(K-1)+1} L_4 D_4^{-1}. \]
Hence $\text{Spec}(A_4)$ is determined by $\text{Spec}(L_4 D_4^{-1})$, or to be more specific, $\lambda \in \text{Spec}(L_4 D_4^{-1}) \Leftrightarrow 1-\frac{1-p}{p(K-1)+1} \lambda \in \text{Spec}(A_4)$.  Furthermore, note that $L_4 D_4^{-1}$ and $D_4^{-1/2} L_4 D_4^{-1/2}$ are similar matrices and share the same spectrum.  It turns out that the spectrum of $D_4^{-1/2} L_4 D_4^{-1/2}$ is bounded within the spectrum of $D_L^{-1/2} L D_L^{-1/2}$, which in turn is equal to $\text{Spec}(L D_L^{-1}) \subset [0,\Lambda]$ by similarity.  Let $g$ be an eigenvector of $D_4^{-1/2} L_4 D_4^{-1/2}$ with eigenvalue $\lambda$ and let $g_1, \ldots, g_j$ be an orthonormal basis of eigenvectors of $D_L^{-1/2} L D_L^{-1/2}$ (such a basis exists since it is a symmetric matrix) with eigenvalues $\mu_1, \ldots, \mu_j$.  We can write
\[ \begin{pmatrix} \textbf{0}_c \\ g \end{pmatrix} = \alpha_1 g_1 + \ldots + \alpha_j g_j \]
for some $\alpha_1, \ldots, \alpha_j$, where $\textbf{0}_c$ is the vector of zeros with length equal to the number of edges classified as $\pm c$.
Then 
\begin{align*} 
\alpha_1 \mu_1 g_1 + \ldots + \alpha_j \mu_j g_j &= D_L^{-1/2} L D_L^{-1/2} \begin{pmatrix} \textbf{0}_c \\ g \end{pmatrix} \\
           &= \begin{pmatrix} D_1^{-1/2} L_1 D_1^{-1/2} & D_1^{-1/2} L_2 D_4^{-1/2} \\ D_4^{-1/2} L_3 D_1^{-1/2} & D_4^{-1/2} L_4 D_4^{-1/2} \end{pmatrix} \begin{pmatrix} \textbf{0}_c \\ g \end{pmatrix} \\
           &= \begin{pmatrix} D_1^{-1/2} L_2 D_4^{-1/2}g \\ D_4^{-1/2} L_4 D_4^{-1/2}g \end{pmatrix} \\
           &= \begin{pmatrix} D_1^{-1/2} L_2 D_4^{-1/2}g \\ \lambda g \end{pmatrix}.
\end{align*}
Taking the Euclidean norm of the beginning and ending expressions, we see that
\begin{align*}
\abs{\alpha_1 \mu_1} + \ldots + \abs{\alpha_j \mu_j} &= \norm{\begin{pmatrix} D_1^{-1/2} L_2 D_4^{-1/2}g \\ \lambda g \end{pmatrix}}_2 \\
  &\geq \norm{\lambda g}_2 \\
  &=\lambda(\abs{\alpha_1} + \ldots + \abs{\alpha_j}). 
\end{align*}
Because we assumed that $\mu_i \in [0,\Lambda]$ for all $i$, it would be a contradiction if $\lambda < 0$ or $\lambda > \Lambda$.  The case $\lambda = 0$ is possible if and only if there is a vector in $\ker L$ that is supported on the unlabelled edges.  To see this, note that if $\lambda = 0$ then 
\begin{align*}
\alpha_1^2 \mu_1 + \ldots + \alpha_j^2 \mu_j &= \begin{pmatrix} \textbf{0}_c \\ g \end{pmatrix}^T D_1^{-1/2} L_2 D_4^{-1/2} \begin{pmatrix} \textbf{0}_c \\ g \end{pmatrix} \\
  &= \begin{pmatrix} \textbf{0}_c \\ g \end{pmatrix}^T \begin{pmatrix} D_1^{-1/2} L_2 D_4^{-1/2}g \\ \lambda g \end{pmatrix} \\
  &= 0 
\end{align*}
which implies $\alpha_i \mu_i = 0$ for all $i$ and therefore $\left(\begin{smallmatrix} \textbf{0}_c \\ g \end{smallmatrix} \right)\in \ker L$.  Finally, since we assumed that no vector in $\ker L$ is supported on the unlabelled edges and that $\frac{\Lambda - 2}{2K + \Lambda - 2} < p < 1$, we conclude that $\text{Spec}(L_4 D_4^{-1}) \subset (0, \Lambda]$ and therefore $\text{Spec}(A_4) \subset \left[1-\frac{1-p}{p(K-1)+1} \Lambda, 1 \right) \subset (-1, 1)$.

To see that the solution $\hat{f}_\infty^c = (I-A_4)^{-1}A_3 \psi^c$ does not depend on $p$, note that $I-A_4$ is a submatrix of $\frac{1-p}{p(K-1)+1}L D_L^{-1}$ so that $\frac{p(K-1)+1}{1-p}(I-A_4)$ does not depend on $p$.  Then write $\hat{f}_\infty^c$ as
\[ \hat{f}_\infty^c = \left[\frac{p(K-1)+1}{1-p}(I-A_4)\right]^{-1} \times \frac{1}{1-p}A_3 \psi^c \]
and note that $\frac{p(K-1)+1}{1-p}A_3$ is an \textit{off-diagonal} submatrix of $\frac{p(K-1)+1}{1-p}I - L D_L^{-1}$ and therefore does not depend on $p$ either.
\end{proof}

Note that while the limit $f_\infty^c$ exists, the matrix $I-A_4$ could be ill-conditioned.  In practice, it may be better to approximate $f_\infty^c$ with $f_t^c$ for large enough $t$.  Also, the algorithm will converge faster for smaller values of $p$ and if $\hat{f}_0^c = \textbf{0}$.

\subsection{Experiments}
We use some simulations to illustrate how Algorithm \ref{alg1} works.  
\begin{example}
Figure \ref{fig: A_H0} shows a simplicial complex in which a single oriented edge $e_1$ has been labelled with class $c$ (indicated by the red color) and all other edges are unlabelled.  Figure \ref{fig: A_H1} shows what happens when this single label is propagated $T = 100$ steps using Algorithm \ref{alg1} with $L = L_1^\text{up}$, $p = 0.9$, and with $f_0^c$ equal to the indicator function on $e_1$.  After the $T$ steps have been performed the edges are oriented and labelled according to the sign of $f_k^c$ (if $f_k^c(e_i)=0$ for an oriented edge $e_i$, then that edge is left unoriented and unlabelled in the figure).  Figures \ref{fig: A_H2} and \ref{fig: A_H3} show the same thing with $L = L_1^{\text{down}}$ and $L = L_1$, respectively.  The results using $L_1^\text{up}$ and $L_1^{\text{down}}$ have a clear resemblance to magnetic fields.  When $L = L_1^{\text{down}}$, ``mass'' is preserved which creates multiple vortices where the flow spins around a triangle.  The walk using $L_1^\text{up}$ tries to maintain local consistency of motion, creating sources and sinks in the process.  The full $L_1$ walk strikes somewhat of a balance between the two, resulting in a more circular flow with a single vortex in the lower left.
\end{example}

\begin{figure}
\centering
\begin{subfigure}[b]{0.49\textwidth}
\includegraphics[width=\textwidth]{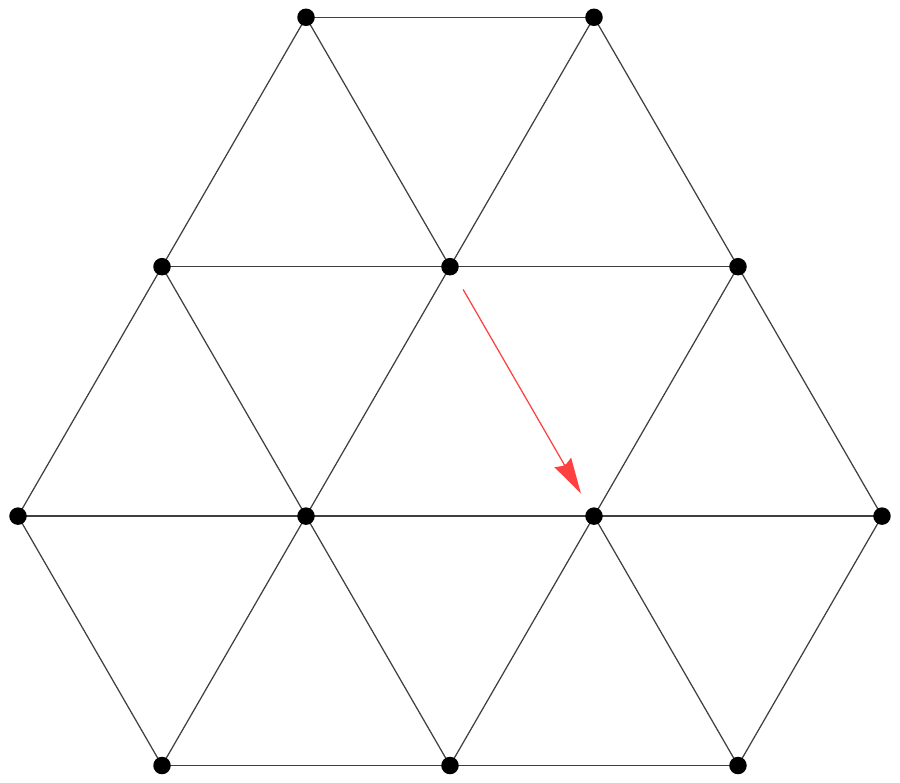}
\caption{A 2-complex with a labelled edge.}
\label{fig: A_H0}
\end{subfigure}
\begin{subfigure}[b]{0.49\textwidth}
\includegraphics[width=\textwidth]{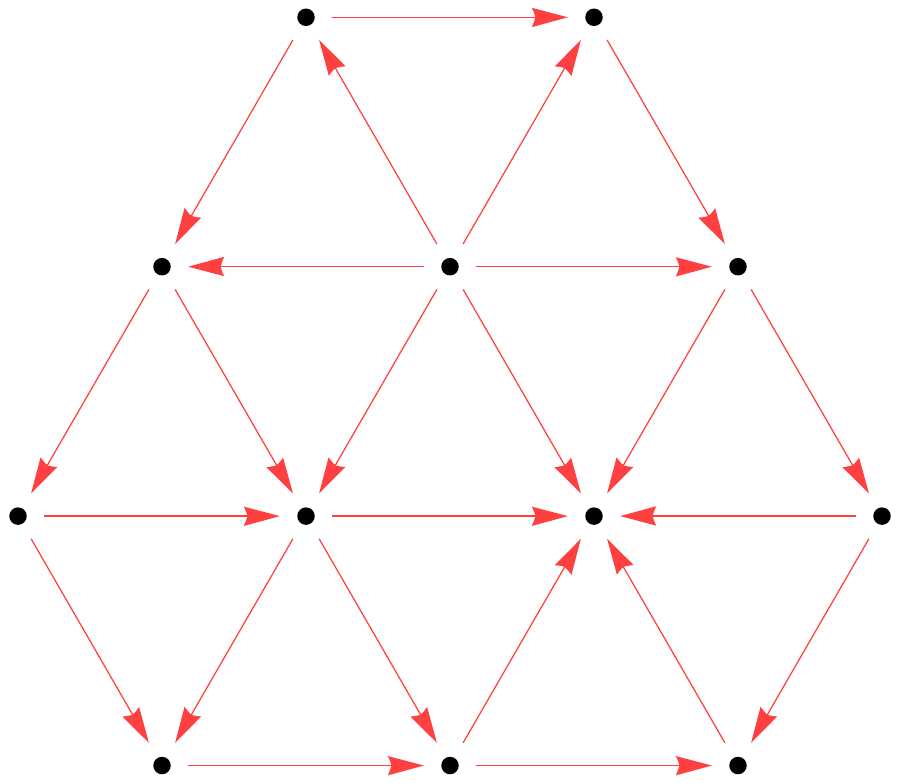}
\caption{Label propagation with $L = L_1^\text{up}$.}
\label{fig: A_H1}
\end{subfigure}
\begin{subfigure}[b]{0.49\textwidth}
\includegraphics[width=\textwidth]{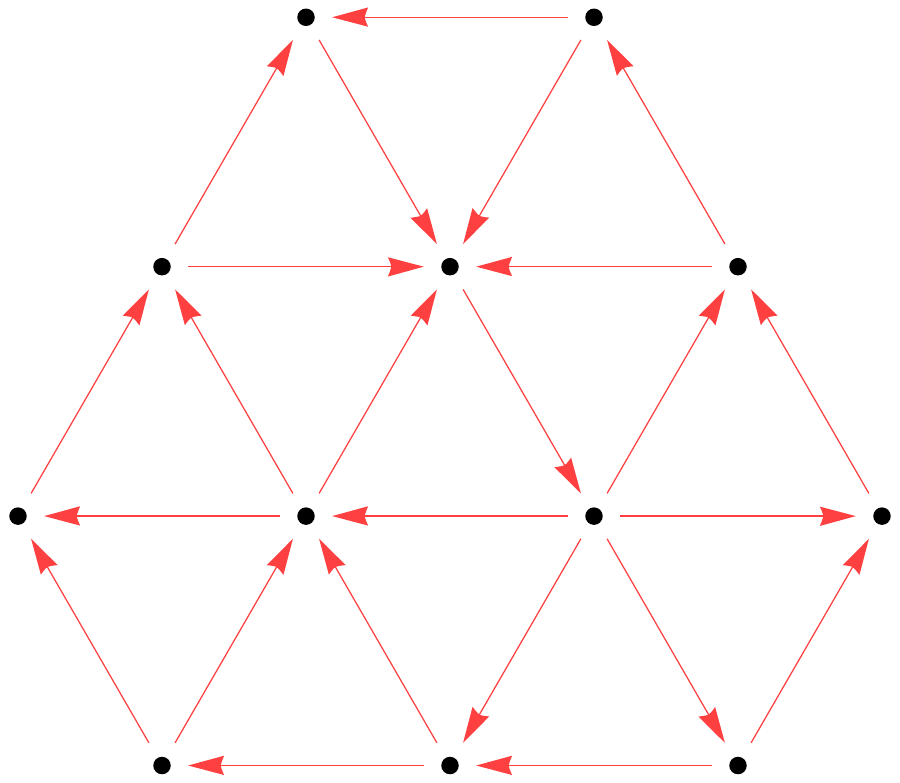}
\caption{Label propagation with $L = L_1^\text{down}$.}
\label{fig: A_H2}
\end{subfigure}
\begin{subfigure}[b]{0.49\textwidth}
\includegraphics[width=\textwidth]{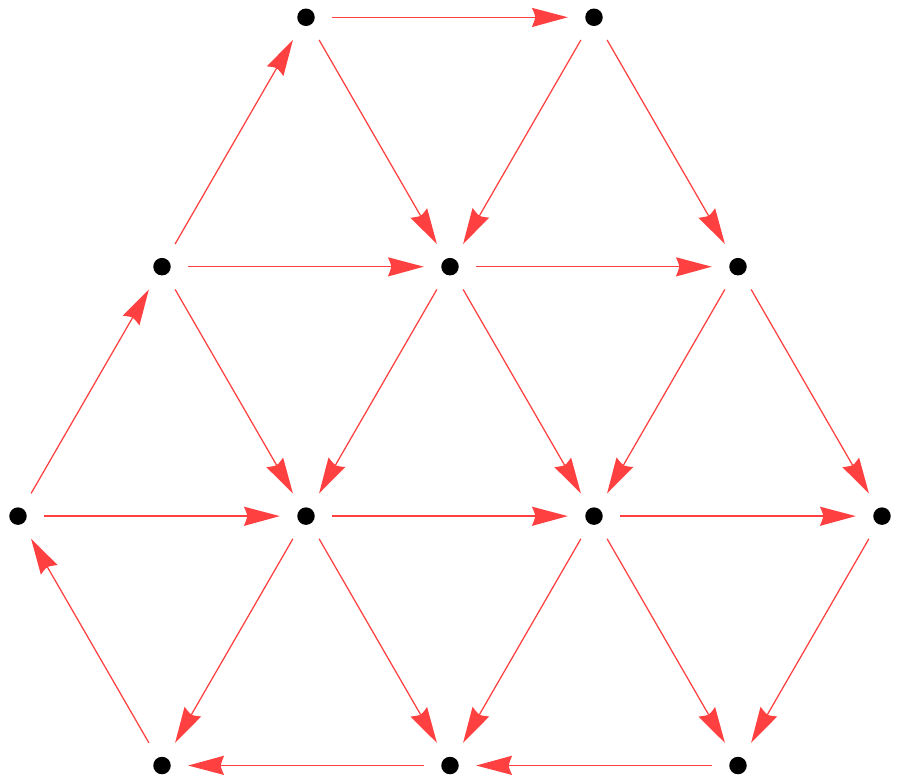}
\caption{Label propagation with $L = L_1$.}
\label{fig: A_H3}
\end{subfigure}
\caption{Edge label propagation with 1 class.}
\label{fig: A_H}
\end{figure}

\begin{example}\label{ex: L1}
Figure \ref{fig: B_H0} shows a simplicial complex in which two edges have been labelled with class $c=1$ (indicated by the red color) and two more edges have been labelled with class $c=2$ (indicated by the blue color).  Figure \ref{fig: B_H1} shows what happens when the labels are propagated $T=1000$ steps using Algorithm \ref{alg1} with $L = L_1$, $p = 0.9$, and $f_0^c$  equal to the indicator function on the oriented edges labelled with classes $c=1,2$.  Every edge is then oriented and labelled according to the sign of $f_{T}^{c=1}$, if $|f_{T}^{c=1}| > |f_{T}^{c=2}|$, or $f_{T}^{c=2}$, if $|f_{T}^{c=1}| < |f_{T}^{c=2}|$. Notice that only a small number of labels are needed to induce large-scale circular motion.  Near the middle, a few blue labels mix in with the red due to the asymmetry of the initial labels.
\end{example}

\begin{figure}
\centering
\begin{subfigure}[b]{\textwidth}
\centering
\includegraphics[width=0.85\textwidth]{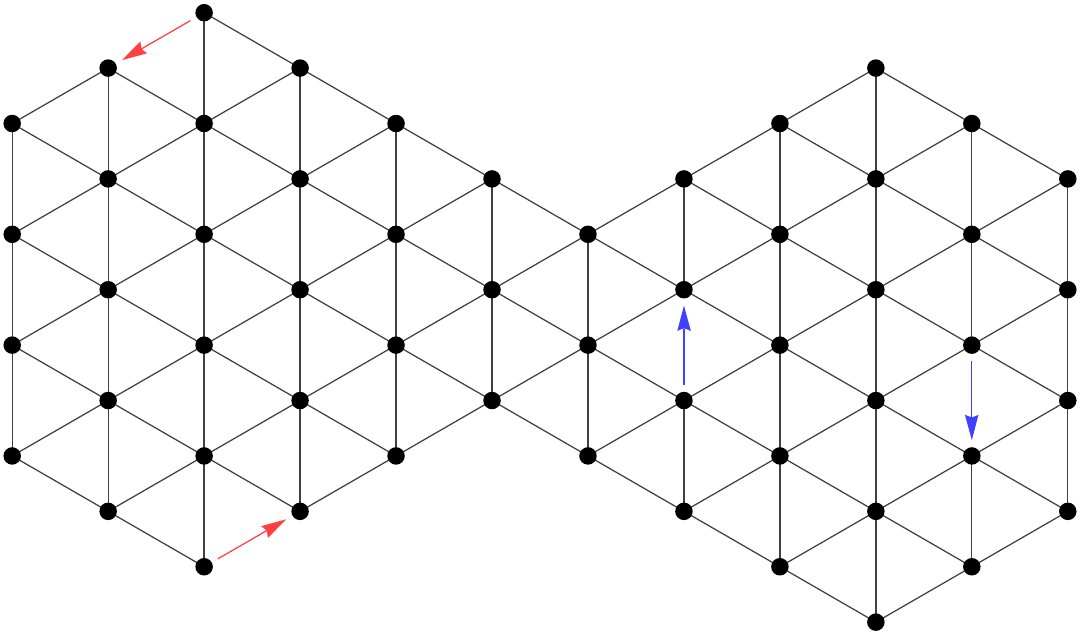}\\
\caption{A 2-complex with two different labels on four edges.\vspace{.2in}}
\label{fig: B_H0}
\end{subfigure}
\begin{subfigure}[b]{\textwidth}
\centering
\includegraphics[width=0.85\textwidth]{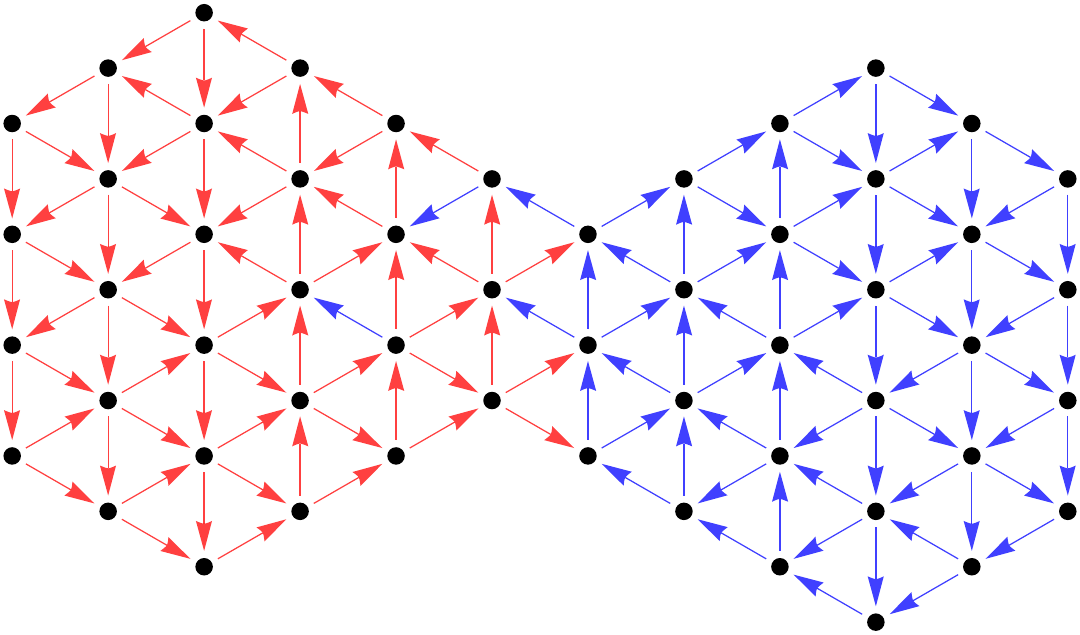}\\
\caption{Label propagation with $L = L_1$.}
\label{fig: B_H1}
\end{subfigure}
\caption{Edge label propagation with two classes.}
\label{fig: B_H}
\end{figure}

\section{Discussion}
In this paper, we introduced a random walk with absorbing states on simplicial complexes.  Given a simplicial complex of dimension $d$, the
relation between the random walk and the spectrum of the $k$-dimensional Laplacian for $1 \leq k \leq d$ was examined. We compared the 
Dirichlet random walk we introduced to the Neumann random walk introduced in Rosenthal and Parzanchevski \cite{parzanchevski2012simplicial}.

There remain many open questions about random walks on simplicial complexes and the spectral theory of higher order Laplacians. Possible future
directions of research include:
\begin{enumerate}
\item[(1)] Is there a Brownian process on a manifold that corresponds to the continuum limit of these new random walks?
\item[(2)] Is it possible to use conditioning techniques from stochastic processes such as Doob's $h$-transform to analyze these walks?
\item[(3)] What applications do these walks have to problems in machine learning and statistics?
\end{enumerate}

\section*{Acknowledgements}
SM would like to thank Anil Hirani, Misha Belkin, and Jonathan Mattingly for useful comments. SM is pleased to acknowledge support from grants NIH (Systems Biology): 5P50-GM081883, AFOSR: FA9550-10-1-0436, and NSF CCF-1049290.  JS would like to thank Kevin McGoff for proofreading and useful comments.  JS is pleased to acknowledge support from NSF grants DMS-1045153 and DMS-12-09155.

\bibliographystyle{plain}
\bibliography{bibliography}

\end{document}